\documentclass[a4,12pt]{article}%
\usepackage{amsmath}%
\usepackage{tikz}
\usepackage{amsfonts}%
\usepackage{amssymb}%
\usepackage{graphicx}
\usepackage{amssymb, amsmath, latexsym, mathrsfs, verbatim, calc}
\usepackage{amsthm}
\usepackage{amssymb}

\definecolor{ProcessBlue}{cmyk}{1,0,0,0.40}

\newtheorem{theorem}{Theorem}

\newtheorem{corollary}[theorem]{Corollary}

\newtheorem{definition}{Definition}

\newtheorem{lemma}{Lemma}

\newtheorem{proposition}{Proposition}
\newtheorem{remark}[theorem]{Remark}

\def\ds{\displaystyle}
\newcommand{\be}{\begin{equation}}
\newcommand{\ee}{\end{equation}}
\def\rg{\rangle}
\def\lg{\langle}

\newcommand{\calM}{{\cal M}}
\newcommand{\ep}{\varepsilon}
\newcommand{\dN}{{{\bf N}}}
\newcommand{\dR}{{{\bf R}}}
\newcommand{\E}{{{\bf E}}}

\newcommand{\prob}{{{\bf P}}}
\newcommand{\Prob}{{{\bf P}}}

\newcommand{\calF}{\mathcal{F}}

\newcommand{\calH}{\mathcal{H}}

\newcommand{\calS}{\mathcal{S}}

\newcounter{figurecounter}
\title{ Markov games with frequent actions and incomplete information}
\author{P. Cardaliaguet\thanks{Ceremade, Universit\'e Paris-Dauphine,
Place du Mar\'echal de Lattre de Tassigny, 75775 Paris cedex 16 (France). cardaliaguet@ceremade.dauphine.fr. }
 \and Catherine Rainer\thanks{Universit\'e de Bretagne Occidentale, 6, avenue Victor-le-Gorgeu, B.P. 809, 29285 Brest cedex, France.
e-mail: Catherine.Rainer@univ-brest.fr} \and Dinah Rosenberg\thanks{D\'epartement ESD, HEC Paris. rosenberg@hec.fr}
  \and Nicolas Vieille\thanks{D\'epartement ESD, HEC Paris. vieille@hec.fr}}

\begin{document}

\maketitle

\begin{abstract}
We study a two-player, zero-sum,  stochastic game with incomplete information on one side in which the players are allowed to play more and more frequently. The informed player observes the realization of a Markov chain on which the payoffs depend, while the non-informed player only observes his opponent's actions. We show the existence of a limit value as the time span between two consecutive stages vanishes;  this value is characterized through an auxiliary optimization problem and as the solution of an Hamilton-Jacobi equation.
\end{abstract}

\vspace{3mm}

\vspace{3mm}

\noindent{\bf A.M.S. classification :} 91A05, 91A15, 60J10
\vspace{3mm}

\bigskip

\section{Introduction}

This paper contributes to the expanding literature on dynamic games with asymmetric information, in which information parameters change with time, see e.g. Athey and Bagwell (2008), Mailath and Samuelson (2001), Phelan (2006), Renault (2006), Wiseman (2008), Neyman (2008) and more recently Escobar and Toikka (2013) and Renault, Solan and Vieille (2012). In these papers, payoff-relevant types are private information,  and follow  Markov processes.

The mathematical analysis of such games for a fixed discount factor $\delta$ remains beyond reach, and all of the above papers (as well as most of the literature on repeated games, see Mailath and Samuelson (2006)) focus on the limiting case where $\delta \to 1$, with the interpretation that players are "very patient". Yet, another, equally interesting interpretation which is consistent with $\delta \to 1$ is that the players get the opportunity to play very frequently. That the two interpretations may lead to sharpingly contrasted results was first pointed in Abreu, Milgrom and Pearce (1991) for repeated games with imperfect monitoring, see also Fudenberg and Levine (2007) for a recent elaboration on this issue. Recently, this point was also convincingly made in Peski and Wiseman (2012), which analyzes stochastic games with frequent actions, and which contrasts results with those proven in H\"orner, Sugaya, Takahashi and Vieille (2011)  for stochastic games with patient players.

We here adhere to this alternative interpretation. Our goal is to clear the mathematical problems in the analysis of the benchmark case of two-player, zero-sum games, and we revisit the model of Renault (2006) as follows. The interaction between two players is affected by a payoff-relevant type, which evolves in continuous time according to a Markov process $(s_t)_{t\geq 0}$. The two players choose actions at discrete time instants (called stages). Together with the current type, these actions determine the (unobserved) payoff of player 2 to player 1. The realizations of $(s_t)$ are observed by player 1, but not by player 2, who only observes past action choices of player 1. Players discount future payoffs at a fixed and common discount rate $r>0$, and the time span between two consecutive stages is $\frac{1}{n}$.

We prove the existence and provide a characterization of the limit value, as the time span $\frac{1}{n}$ goes to zero. While our setup is directly inspired from Renault (2006), our analysis is significantly different. In Renault (2006), the transition rates between any two consecutive stages remain constant when players get more patient. (At least when the process of types is irreducible,) the initial private information of a player has a finite lifetime,  and the limit value does not depend on the initial distribution. Here instead, transitions rates are of the order of $\frac{1}{n}$: as players play more often, the probability that the state changes from one stage to the next vanishes. As a result, the limit value does depend on the initial distribution.

We first analyze the case of exogenous transitions (transition rates do not depend on action choices). Adapting techniques from the literature on repeated games with incomplete information, see Aumann and Maschler (1995), we give a semi-explicit formula of the limit value as the value of an auxiliary optimization problem, which we use to get explicit formulas in a number of cases.
Using PDE techniques, we provide an alternative characterization of the limit value as the unique solution (in a weak sense) to a non-standard Hamilton-Jacobi (HJ) equation. This equation can be understood as the infinitesimal counterpart of a dynamic programming principle.

We next expand significantly this framework to allow first for endogenous transitions (transition rates do depend on actions) and next, for incomplete information on both sides (each player observes and controls his own Markov chain). In both settings we show that the limit value exists and is characterized as the unique (weak) solution of a HJ equation. Our techniques for this analysis  (viscosity solutions of Hamilton-Jacobi equations, passage to the limit in these equations) are reminiscent of the ones developed for  differential games, as in Evans and Souganidis (1984). However, because of the information asymmetry, the Hamilton-Jacobi equation satisfied by the limit value takes the form of an obstacle problem, much as  in the case of  differential games with incomplete information (Cardaliaguet and Rainer (2009a), Cardaliaguet (2009)), yet with a significant difference. Indeed,  information is here disclosed to the informed player(s) through time (and not only at the initial instant); this leads to a new HJ equation and to a  slightly different definition of weak solution (cf. the discussion after Definition  \ref{defvisco} and at the beginning of  the proof of the comparison principle). The passage from discrete games to continuous equations partially relies on methods developed for repeated games in Vieille (1992), Laraki (2002), Cardaliaguet, Laraki and Sorin (2012), and the starting point of our analysis for incomplete information on both sides is inspired by Gensbittel and Renault (2012).

The paper is organized as follows. We first present the model and state the main results (Section \ref{sec_res}), which we illustrate through several examples in Section \ref{sec:exemple} and which we prove in Section \ref{sec:proofmain}.
Games with  endogenous  transitions are analyzed in Section \ref{sec:nonendo}, while Section \ref{sec:bothsides} is devoted to games with incomplete information on both sides. In the appendix we collect the proofs of several technical facts, including a new comparison principle adapted to our framework.

\section{Model and main result}\label{sec_res}

\subsection{Model}

We start with the simpler version of the model. There is a finite set of states $S$ of  cardinal $|S|$. With each state is associated a zero-sum game with finite action sets $A$ and $B$ and payoff function $g(s,\cdot,\cdot)$, where $g:S\times A\times B\to \dR$. Time is continuous, and the state $s_t$ at time $t\geq 0$ follows a Markov chain with law $\prob$, initial distribution $p\in \Delta(S)$ and generator $R=(\rho_{ss'})_{s,s'\in S}$. For $s\neq s'$, $\rho_{ss'}$ is thus the rate of transitions from $s$ to $s'$, while $-\displaystyle \rho_{ss}=\sum_{s'\neq s}\rho_{ss'}$ is the rate of transitions out of state $s$. We denote by $P(\cdot)$ the transition semi-group of $(s_t)_{t\geq 0}$, so that $P_{h}(s,s')=\prob(s_{t+h}=s'\mid s_t=s)$ for all $t,h\geq 0$ and $s,s'\in S$. The map $t\mapsto P_t$ is a solution to the Kolmogorov equation $P'_t=RP_t$, and is given by $P_t=\exp(tR)$.
\bigskip

Given  $n\in \dN^*$, we let $G_n(p)$ denote the following, two-player game with infinitely many stages. In each stage $k\in \dN$, players choose actions $a_k$ and $b_k$ in $A$ and $B$, and the payoff is given by $g(s_k^{(n)},a_k,b_k)$ where $s_k^{(n)}:=s_{k/n}$ is the state at time $\frac{k}{n}$. Along the play, player 1 observes past and current realizations of the states $s_k^{(n)}$ and both players observe past actions of player 1, but payoffs are not observed.\footnote{Whether or not actions of player 2 are observed is irrelevant.}

We view $G_n(p)$ as the discretized version of a continuous-time game, with stage $k$ of $G_n(p)$ taking place at physical time $\frac{k}{n}$. As $n$ increases, the time span between two consecutive stages shrinks and the players get the option to play more and more frequently. In physical time, players discount future payoffs at the fixed, positive rate $r>0$. Hence, the weight of stage $k$ in $G_n(p)$ is $\displaystyle \int_{k/n}^{(k+1)/n}re^{-rt}dt=\lambda_n(1-\lambda_n)^k$, where $\lambda_n:=\displaystyle 1-e^{-r/n}$. Note that $\lambda_n\to 0$ as $n\to +\infty$ (and $1-\lambda_n$ may be interpreted as the discount factor between two consecutive stages in $G_n(p)$).

We denote by $\tilde v_n(p)$ the value of the game $G_n(p)$.\footnote{We abstain from using the notation $v_n(p)$, which is associated with games  with $n$ stages.} From the perspective of the literature on repeated games, the game $G_n(p)$ is thus a discounted game, with discount factor $1-\lambda_n$.

%
%

\subsection{Results}

Our main result is the existence of $\lim_{n\to +\infty}\tilde v_n(p)$, together with different characterizations of the limit. We need a few definitions.

Define $\calS(p)$ to be the set of adapted, c\`adl\`ag processes $(p_t)_{t\geq 0}$, defined on some filtred probability space $(\Omega,\calF,\prob,(\calF_t)_{t\geq 0})$, with values in $\Delta(S)$,  and such that, for each $t,h\geq 0$, one has
\be\label{Sp}\E[p_{t+h}\mid \calF_t]=^T\!P(h)p_t, \ \prob-\mbox{a.s.}\ee

Given $\tilde p\in \Delta(S)$, we denote by $u(\tilde p)$ the value of the one-shot, zero-sum game $\Gamma(\tilde p)$ with action sets $A$ and $B$ and payoff function
\[g(\tilde p,a,b):=\sum_{s\in S}\tilde p(s)g(s,a,b).\]
That is, $\ds u(\tilde p)=\max_{x\in \Delta(A)}\min_{y\in \Delta(B)}g(\tilde p,x,y)=\min_{y\in \Delta(B)}\max_{x\in \Delta(A)}g(\tilde p,x,y)$.

\begin{theorem}\label{th1}
The sequence $(\tilde v_n(\cdot))_{n\in \dN}$ converges uniformly, and \textbf{P1} and \textbf{P2} hold, with
$v(p)=\lim_{n\to \infty}\tilde v_n(p) $.
\begin{description}
\item[P1]
$\displaystyle v(p)=\max_{(p_t)\in \calS(p)} \E\left[ \int_0^{+\infty} re^{-rt}u(p_t)dt\right].$
\item[P2] $v(\cdot)$ is the unique viscosity solution of the equation
\be\label{HJ1}
\min\left\{ rv(p)+H(p,Dv(p))\ ; \ -\lambda_{\max}(p, D^2v(p))\right\}=0 \qquad {\rm in }\; \Delta(S),
\ee
where $H(p,\xi)=-\langle ^T\!Rp,\xi\rangle -ru(p)$.
\end{description}
\end{theorem}

\bigskip

Few comments are in order.
We first comment on  \textbf{P2}, and on the intuitive content of equation (\ref{HJ1}). Assuming $v(\cdot)$ (extended to a neighborhood of $\Delta(S)$) is a smooth function, $Dv(p)$ and $D^2v(p)$ stand respectively for the gradient and Hessian matrix of $v(\cdot)$ at $p$, while, loosely speaking, $\lambda_{\max}(p, D^2v(p))$ is the maximal eigenvalue of the restriction of $D^2v(p)$ to the tangent space of $\Delta(S)$ (all formal definitions will be provided later).
According to  \eqref{HJ1},
\begin{itemize}
\item[(i)] $-\lambda_{\max}(p, D^2v(p))\geq 0$, so that the limit value $v$ is concave.
This concavity property---which actually holds for each $\tilde v_n$ (thanks to the so-called splitting results, such as Propositions 2.2 and 2.3 in Sorin (2002))---can also be established using \textbf{P1}.
 \item[(ii)] the inequality $rv(p)+H(p,Dv(p))\geq 0$ always holds on $\Delta(S)$, with equality at any point $p$  where $v$ is strictly concave (or, more precisely, at which $-\lambda_{\max}(p, D^2v(p))> 0$).

It turns out that the Hamilton-Jacobi equation
\[ rw(p)+H(p,Dw(p))=0, \; p\in\Delta(S)\]
characterizes the limit value of the auxiliary game in which no  player observes $(s_t)$ -- the PDE  actually being the infinitesimal version of the dynamic programming equation.
In our game, the equality $rv(p)+H(p,Dv(p))=0$ must intuitively therefore hold wherever it is optimal for
player~1 not to disclose  information. For this reason, the set
\be\label{defsetH}
\calH:=\{p\in \Delta(S)\mid rv(p)+H(p,Dv(p))=0\}
\ee is called the \textit{non-revealing set}.
\end{itemize}

 In general however, one cannot hope  the limit value $v(\cdot)$ to be  smooth. For this reason, the interpretation of the equation \eqref{HJ1} is in the viscosity  sense, see Definition \ref{defvisco} in  Section \ref{sec:nonendo}.

 \bigskip

To illustrate \textbf{P1}, let us specialize  Theorem \ref{th1} to the case where $R$ is identically 0. In such a case, the state $s_0$ is drawn at time 0 and remains fixed throughout time. The game $G_n(p)$ thus reduces to a truly repeated game with incomplete information \`a la Aumann and Maschler (1995). Note that $\calS(p)$ is then equal to the set of c\`adl\`ag martingales with values in $\Delta(S)$ and initial value $p$.

Consider two $\Delta(S)$-valued processes $(p_t)_{t\geq 0}$ and $(\tilde p_\tau)_{\tau\in [0,1]}$, such that
$\tilde p_\tau=p_{-\frac{\ln(1-\tau)}{r}}$ or equivalently $p_t=\tilde p_{1-e^{-rt}}$ a.s., for each $t\geq 0$ and  $\tau\in [0,1]$. Observe that $(p_t)$ is a martingale iff $(\tilde p_\tau)_{\tau\in [0,1]}$ is a martingale, and
\[\E\left[ \int_0^\infty re^{-rt}u(p_t)dt\right]=\E\left[\int_0^1u(\tilde p_\tau)d\tau \right].\]
Therefore, denoting by $\calM_{[0,1]}(p)$ the set of c\`adl\`ag martingales defined over $[0,1]$, with values in $\Delta(S)$ and  starting from $p$, one has
\begin{equation}\label{splitting}
v(p)=\max_{(\tilde p_\tau)\in \calM_{[0,1]}(p)}\E\left[ \int_0^1u(\tilde p_\tau)d\tau \right],
\end{equation}
a well-known formula for repeated games (see Section 3.7.2 in Sorin (2002)). In a sense, the assertion \textbf{P1} thus provides the appropriate generalization of (\ref{splitting}) to the case of an arbitrary  transition rate matrix $R$.\\
\bigskip

\textbf{P1} and \textbf{P2} provide two alternative (and independent) characterizations of the limit value, as the value of an auxiliary optimization problem, and as a solution to a Hamilton-Jacobi PDE. We next state a verification theorem, which relies on \textbf{P2} to give a sufficient condition under which a process in $\calS(p)$ is optimal in \textbf{P1}.

\begin{theorem}\label{thm:sufficientcond} Assume that  $v$ is of class $C^2$ in a neighborhood of $\Delta(S)$. Let $p\in \Delta(S)$ and $(p_t)\in \calS(p)$ be given, and assume that (i) and (ii) below hold.
\begin{itemize}
\item[(i)]  $\prob$-a.s., one has $p_s\in{\mathcal H}$ and $v(p_s)-v(p_{s-})=\langle Dv(p_{s-}),p_s-p_{s-}\rangle$ for all $s\geq 0$,

\item[(ii)]  $(p_t)$ has no continuous martingale part.
\end{itemize}
Then $(p_t)$ achieves the maximum in \textbf{P1}.
\end{theorem}

\begin{remark}{\rm  As $R$ is the generator of the transition semi-group $P$, (\ref{Sp}) implies that each process  $(p_t)\in\calS(p)$  can be decomposed,  $\Prob$-a.s. for all $t\geq 0$, as
\be
\label{decompo}
 p_t=p+\int_0^t\; ^T\! Rp_sds+m_t,
 \ee
where $(m_t)$ is a martingale in the filtration generated by $(p_t)$. This martingale itself can be decomposed into a continuous and a purely discontinuous part (see Protter (2005)).
}\end{remark}

The most important condition in Theorem \ref{thm:sufficientcond} is (i), which states that the ``information process" $(p_t)$ must live in the non-revealing set ${\mathcal H}$ and can jump only on the flat parts of the graph of the limit value $v(\cdot)$; this condition is known to be sufficient in a class of simpler games, such as in Cardaliaguet and Rainer (2009b).  Condition (ii) is often satisfied in practice, as  the examples in the next section show.

\begin{proof}
We write the It\^o formula for $e^{-rt}v(p_t)$, using the decomposition \eqref{decompo}:
\[\begin{array}{rr}
\ds e^{-rt}v(p_t)=&\ds v(p)-\int_0^tre^{-rs}v(p_s)ds+\int_0^te^{-rs}Dv(p_{s-})dp_s+\frac 12\int_0^te^{-rs}D^2v(p_s)d\langle m^c\rangle_s\\
&\ds +\sum_{0<s\leq t}e^{-rs}\left(v(p_s)-v(p_{s-})-\langle Dv(p_{s-}),p_s-p_{s-}\rangle\right),
\end{array}\]
where $(m^c_t)$ is the continuous part of the martingale $(m_t)$ and $\langle m^c\rangle$ its quadratic variation.\\
Under the assumptions {\em (i)} and {\em(ii)}, the two last terms in this equation vanish.
Then, replacing $(p_t)$ by its martingale decomposition and taking  expectations on both sides, we get
\[ e^{-rt}\E[v(p_t)]=v(p)-\E\left[\int_0^te^{-rt}\left(rv(p_s)-\langle \;^T\!Rp_s,Dv(p_s)\rangle\right)ds\right].\]
It is time now to apply the assumption $\ p_s\in{\mathcal H} \ $, which   leads  to
\[ e^{-rt}\E[v(p_t)]=v(p)-\E\left[\int_0^tre^{-rt}u(p_s)ds\right].\]
The result follows when letting $t\to +\infty$.
\end{proof}




\section{Examples and Applications}\label{sec:exemple}

  We here illustrate how \textbf{P1} can be used to provide explicit formulas for $v(p)$ in various cases. This section is organized as follows. We first provide in Lemmas \ref{lemmupper} and \ref{lemmlower} respectively upper and lower bounds on $v(\cdot)$ which always hold. We next identify several cases where these bounds coincide, thereby pinning down $v(\cdot)$. We finally discuss two examples in more detail.

  \subsection{Upper and lower bounds for the limit value}

Let $(p^*_t)$ be defined by $p^*_t=^T\!P_tp$. The process $(p^*_t)$ is the unique deterministic process in $\calS(p)$. It is the process of beliefs held by player 2 when player 1 plays in a non-revealing manner (that is, ignores his own private information) or  equivalently, the beliefs of an outside observer who would not observe the informed player's actions. Observe that $\E[p_t]=p^*_t$ for every $t\geq 0$ and $(p_t)\in \calS(p)$.

We denote by $\textrm{cav }u:=\Delta(S)\to \dR$ the concavification of $u(\cdot)$.

\begin{lemma}\label{lemmupper}
One has
\[v(p)\leq \int_0^\infty re^{-rt}\mathrm{cav}\ u(p^*_t)dt.\]
\end{lemma}
\begin{proof}
For any $(p_t)\in \calS(p)$, one has
\[\E\int_0^\infty re^{-rt}u(p_t)dt= \int_0^\infty re^{-rt}\E[u(p_t)]dt \leq  \int_0^\infty re^{-rt}\textrm{cav}\ u(\E[p_t])dt=  \int_0^\infty re^{-rt}\textrm{cav}\ u(p^*_t)dt,\]
where the first equality follows from Fubini Theorem, and the first inequality follows from the inequality $u\leq \textrm{cav} \ u$ and from Jensen inequality.
\end{proof}
\bigskip

 For $s\in S$, we denote by $\delta_s\in \Delta(S)$ the probability measure which assigns probability one to $s$.

\begin{lemma}\label{lemmlower}
One has
\[v(p)\geq \int_0^\infty re^{-rt}u(p^*_t)dt,\]
and
\[v(p)\geq  \sum_{s\in S}u(\delta_s)\int_0^\infty re^{-rt}p^*_t(s)dt.\]
\end{lemma}

\begin{proof}
These lower bounds for $v(p)$ are obtained when computing $\displaystyle\E\int_0^\infty re^{-rt}u(p_t)dt$ for specific processes $(p_t)\in \calS(p)$.

The first lower bound is obtained when setting $p_t:=p^*_t$. Intuitively, the right-hand side is then the amount which is secured by the strategy which plays at each $t$ an optimal (non-revealing) strategy in the average game associated with the current belief of player 2.

The second lower  bound obtains when setting $p_t:=\delta_{s_t}$. Indeed, one then has
\begin{eqnarray*}\E\int_0^\infty re^{-rt}u(p_t)dt&=& \int_0^\infty re^{-rt}\E[u(\delta_{s_t})]dt \\
&=& \sum_{s\in S} u(\delta_s)\int_0^\infty re^{-rt}\prob(s_t=s)dt \\
&=& \sum_{s\in S} u(\delta_s)\int_0^\infty re^{-rt}p^*_t (s)dt.
\end{eqnarray*}
Intuitively, the right-hand side is then the amount which is secured by a strategy which would announce at each $t$ the current state, and then play optimally in the corresponding game.
\end{proof}

\begin{corollary}\label{concaveconvex}
If $u$ is concave, then
$v(p)=\displaystyle \int_0^\infty re^{-rt}u(p^*_t)dt$.

  If $u$ is convex, then $v(p)=\displaystyle  \sum_{s\in S}u(\delta_s)\int_0^\infty re^{-rt}p^*_t(s)dt$.
\end{corollary}

\begin{proof}
If $u$ is concave, then $u=\textrm{cav}\  u$, and the result follows from Lemmas \ref{lemmupper} and \ref{lemmlower} (first lower bound).

If $u$ is convex, then $\textrm{cav}\ u(p)=\sum_{s\in S}p(s)u(\delta_s)$ for each $p\in \Delta(S)$, and the result again follows from Lemmas \ref{lemmupper} and \ref{lemmlower} (second lower bound).
\end{proof}

We now illustrate in  these two simple cases the alternative characterization  {\bf P2} in Theorem \ref{th1}.
If $u$ is smooth and concave, then the map $$w(p):= \int_0^\infty re^{-rt}u(p^*_t)dt=  \int_0^\infty re^{-rt}u(^T\!P_tp)dt$$ is concave and  satisfies
\[
\ds rw(p)+H(p,Dw(p)) \; =  \ds rw(p)-\langle ^T\!Rp,Dw(p)\rg -ru(p) =0.\]
Therefore $w$ is a
solution to the Hamilton-Jacobi equation
\be\label{eqHJExemples}
\min\left\{ rv(p)+H(p,Dv(p))\ ; \ -\lambda_{\max}(p, D^2v(p))\right\}=0 \qquad {\rm in }\; \Delta(S).
\ee
By {\bf P2}, this shows (again) that $v=w$.
Recalling the definition of the non-revealing set ${\mathcal H}$ in \eqref{defsetH}, we here have  ${\mathcal H}=\Delta(S)$. As the deterministic process $(p^*_t)$ satisfies  conditions (i) and (ii) of Theorem \ref{thm:sufficientcond}, it is optimal in {\bf P1}: in other words, player 1 does not reveal anything.

Assume instead that $u$ is smooth and convex. Then the map
$$
 w(p):=  \sum_{s\in S}u(\delta_s)\int_0^\infty re^{-rt}p^*_t(s)dt
 $$
satisfies
$$
rw(p)+H(p,Dw(p))   = r\sum_s p_s u(\delta_s) -ru(p) \geq 0,
$$
because $u$ is convex. As $D^2w(p)=0$, $w$ solves \eqref{eqHJExemples} (actually one has to be more cautious here and to use the notion of viscosity solution of Definition \ref{defvisco}). From this and {\bf P2}, it follows that $v=w$.
Moreover, the non-revealing set is given by ${\mathcal H}= \{p: \ u(p)= {\rm cav}(u)(p)\}$ and thus contains $ S$. Then the process $(p_t:=\delta_{s_t})$ satisfies conditions (i) and (ii) of Theorem \ref{thm:sufficientcond}, so that it is optimal in {\bf P1}: player 1 reveals all his information.
\bigskip

In both the concave and the convex cases, $v(p)$ is given by $\displaystyle \int_0^\infty re^{-rt}\textrm{cav}\ u(p^*_t)dt$. We show that this latter formula is valid in many cases beyond the concave and convex case, but not always.

\subsection{Two-state games}

In the following, we focus on the case where $S:=\{s_1,s_2\}$ contains only two states, and we identify a probability measure over $S$ with the probability assigned to state $s_1$. In particular, $u$ will be viewed as a function defined over $[0,1]$.
We denote by $p^*_\infty:=\lim_{t\to \infty}p^*_t\in [0,1]$ the unique invariant measure of $(s_t)$, and let $\underline{p},\bar p \in [0,1]$ be such that $\underline p\leq p^*_\infty\leq \bar p$, and
$(p^*_\infty,\textrm{cav}\ u(p_\infty^*) )= \alpha (\underline{p},u(\underline{p}))+(1-\alpha) (\bar p,u(\bar p))$, for some $\alpha\in [0,1]$.

Such distributions $\underline{p}$ and $\bar p$ always exist, but need not to be uniquely defined.

\begin{lemma} Assume $\underline{p}<p^*_\infty<\bar p$.
One has \begin{equation}\label{formule}v(p)=\int_0^\infty re^{-rt}\textrm{cav}\  u(p^*_t)dt\mbox{ for each } p\in  [\underline{p},\bar p].\end{equation}
If moreover the equality $u=\mathrm{cav}\ u$ holds on $[0,\underline{p}]$ (respectively on $[\bar p,1]$), then (\ref{formule}) holds on the interval   $[0,\underline{p}]$ (respectively on $[\bar p,1]$).
\end{lemma}

If we further assume that $\textrm{cav}\ u$ is of class ${\mathcal C}^1$ with $\textrm{cav}\ u(p) >u(p)$ in $(\underline{p}, \bar p)$, then one can easily check that the non-revealing set  defined in \eqref{defsetH} satisfies  $\{\underline{p},\bar p\}\subset {\mathcal H}\subset [0, \underline{p}]\cup [\bar p, 1]$. In particular, \eqref{formule} follows from the construction---in the proof below---of a process $(p_t)\in {\mathcal S}(p)$ such that $p_t\in \{\underline{p},\bar p\}$ a.s.: this process satisfies (i) and (ii) of Theorem \ref{thm:sufficientcond} (because $v$ is affine on $[\underline{p},\bar p]$) and therefore is optimal in {\bf P1}.  If, moreover, the equality $u=\mathrm{cav}\ u$ holds in $[0, \underline{p}]$, then one has $[0, \underline{p}]\cup\{\bar p\}\subset {\mathcal H}$. We show in the proof below that there is a process $(p_t)\in {\mathcal S}(p)$ with $p_t\in [0,\underline{p}]\cup \{\bar p\}$ a.s.: the same arguments as above show that this process is optimal.
\bigskip

\begin{proof}
Define $\theta:=\inf\{t: p^*_t\in [\underline{p},\bar p]\}$, the first time at which the ``average" belief enters the interval $[\underline{p},\bar p]$. Note that $\theta <+\infty$ and that $p^*_t\in [\underline{p},\bar p]$ for every $t\geq \theta$.

The result follows from the fact, proven below, that there is a process $(p_t)\in \calS(p)$ such that $p_t=p_t^*$ for $t\leq \theta$, and $p_t\in \{\underline{p},\bar p\}$ for $t\geq \theta$, $\prob$-a.s.
Indeed, for any such process, one has
\begin{eqnarray*}\E\int_0^\infty re^{-rt}u(p_t)dt&=&  \int_0^\theta re^{-rt}u(p^*_t)dt +\int_\theta^\infty re^{-rt} \E[u(p_t)]dt \\
&=& \int_0^\theta re^{-rt}\textrm{cav}\ u(p^*_t)dt +\int_\theta^\infty re^{-rt}\textrm{cav}\ u(p^*_t)dt,
\end{eqnarray*}
which will conclude the proof of the lemma.

We now construct the process $(p_t)$. For $t\geq \theta$, define $\displaystyle Q_t:=\left(  \begin{array}{cc} q_{11}(t) & 1-q_{11}(t) \\
1-q_{22}(t) & q_{22}(t)\end{array}\right)$ by
\[q_{11}(t)=\frac{1}{\bar p-\underline{p}}\left(  \bar p -\left(\underline{ p}\times p_{11}(t)+(1-\underline{p})p_{21}(t)\right)\right)\] and \[ q_{22}(t)=\frac{1}{\bar p-\underline{p}}\left( 1- \bar p -\left((1-\bar p)p_{22}(t) +\bar p \times p_{12}(t) \right)\right).\]
Intuitively, $Q_t$ is the transition matrix between the two "states" $\underline{p}$ and $\bar p$ induced by $P_t$.  To see why, observe that, when starting from $\underline{p}$, the probability of being in state $s_1$ at time $t$ is $\underline{p} p_{11}(t)+(1-\underline{p})p_{21}(t)$, which is equal to $q_{11}(t) \underline{p} +(1-q_{11}(t)) \bar p$. Similarly, when starting from $\bar p $, the probability of being in $s_1$ at time $t$ is
$\bar p p_{11}(t)+(1-\bar{p})p_{21}(t)=(1-q_{22}(t))\underline{p}+q_{22}(t)\bar p$. An elementary computation using the Kolmogorov equation $P'_t=RP_t$ yields $Q'_t=\tilde R Q_t$, where the rate matrix $\displaystyle \tilde R=\left( \begin{array}{cc} -\tilde \rho_{12} &\tilde \rho_{12} \\
\tilde \rho_{21} &-\tilde \rho_{21}\end{array} \right)$ is given by
\[\tilde \rho_{12}=\frac{(1-\underline{p})\rho_{21}-\underline{p}\rho_{12}}{\bar p-\underline{p}}\mbox{ and }
\tilde \rho_{21}=\frac{\bar p \rho_{12} -(1-\bar p)\rho_{21}}{\bar p-\underline{p}}.\]
Both $\rho_{12}$ and $\rho_{21}$ are positive. Therefore, there is a Markov process $(q_t)_{t\geq \theta}$ with values in $\Delta(\{\underline{p},\bar p\})$, with rate matrix $\tilde R$ and  initial distribution  $q_\theta$  defined by $p_\theta^*=q_\theta(\underline{p}) \underline{p}+q_\theta(\bar p)\bar p$.

By construction, the  process $(p_t)_{t\geq \theta}$ defined by $p_t:= q_t(\underline{p}) \underline{p}+q_t(\bar p)\bar p$ is a Markov process and  $\E[p_{t+h}\mid \calF_t^p]=^TP_hp_t$ for all $t\geq \theta,h\geq 0$.
Set now $p_t=p_t^*$ for $t<\theta$. Then the process $(p_t)_{t\geq 0}$ satisfies the desired properties.  \end{proof}
\bigskip

In all previous cases, the equality $v(p)=\displaystyle \int_0^\infty re^{-rt}\textrm{cav}\  u(p_t^*)dt$ holds. This is however not always the case, as we now show.\\

\noindent{\bf Example:}
Let a game $A$, $B$ and $g:\{s_1,s_2\}\times A\times B$ be such that (i) $u(0)=u(1)=0$, (ii) $u(p)=1$ for $p\in [\frac13,\frac23]$, and (iii) $u$ is strictly convex on each of the intervals $[0,\frac13]$ and $[\frac23,1]$. Assume that transitions are such that $p^*_\infty \in (\frac13,\frac23)$.\footnote{The existence of such a game follows from Proposition 6 in Lehrer and Rosenberg (2003), which does not appear in the published version of the paper, Lehrer and Rosenberg (2010).}

\begin{proposition} For every $p\notin [\frac13,\frac23]$, one has
$\displaystyle v(p)< \int_0^\infty re^{-rt}\textrm{cav}\ u(p^*_t)dt$.
\end{proposition}

\begin{proof}
Fix $p\in [0,\frac13)$ for concreteness.  We argue by contradiction, and assume that
\[\int_0^\infty re^{-rt}\E[u(p_t)]dt=\int_0^\infty re^{-rt}\textrm{cav}\ u(p_t^*)dt,\] for some  process $(p_t)\in \calS(p)$. Since $\E[u(p_t)]\leq \textrm{cav} \ u(p_t^*)$ for all $t$, one has
$\E[u(p_t)]= \textrm{cav} \ u(p_t^*)$ for Leb-a.e. $t\in \dR^*$.

 Let   $\theta :=\inf\{t: p^*_t\geq \frac13\}$. Observe that $\textrm{cav} \ u(p_t^*)=1$ for $t\geq \theta$, and $\textrm{cav} \ u(p_t^*)=3p_t ^*$ for $t\leq \theta$.

For $t<\theta$, the equality $\E[u(p_t)]= \textrm{cav} \ u(p_t^*)$ implies that the law of $p_t$ is concentrated on $\{0,\frac13\}$, with $\prob(p_t=\frac13)=3p_t^*$. As $t\to \theta^-$ we get  $\prob(p_{\theta^-}=\frac13)=1$, so that $p_{\theta-}=p^*_\theta$ a.s.. Then \eqref{Sp} implies that $p_t= (e^{-^T\!R(\theta-t)})p_\theta^*= p^*_t$ a.s. for $t\in[0,\theta]$, which is impossible since $p_t\in \{0, \frac13\}$ a.s..

\end{proof}

Intuitively, maximizing $\E[u(p_t)]$ leads player 1 to disclose information at time $t$ which he later wishes he hadn't disclosed.

\subsection{An explicit example}

We conclude this section by providing an explicit formula for the limit value in an example due to Renault (2006) (see also H\"orner, Rosenberg, Solan and Vieille (2010)). In that example, both players have two actions, and the payoffs in the two states are given by
\[\left(\begin{array}{cc}
1  & 0 \\
0 & 0
\end{array}\right) \mbox{ and  } \left(\begin{array}{cc}
0 & 0 \\
0 & 1
\end{array}\right) \]
Transitions occur at the rate $\pi>0$, so that
$\displaystyle R= \left(\begin{array}{cc}
-\pi  & \pi \\
\pi & -\pi
\end{array}\right)$.
Observe that $R= M\left(\begin{array}{cc}
-2\pi  & 0 \\
0 & 0
\end{array}\right)M^{-1}$, where $M=\left(\begin{array}{cc}
1  & 1 \\
-1 & 1
\end{array}\right)$, so that
\[P_t= e^{tR}= M\left(\begin{array}{cc}
e^{-2\pi t}  & 0 \\
0 & 1
\end{array}\right)M^{-1}= \frac12 \left(\begin{array}{cc}
1 & 1 \\
1 & 1
\end{array}\right) +\frac{e^{-2t\pi}}{2}\left(\begin{array}{cc}
1 & -1 \\
-1 & 1
\end{array}\right).\]

Note that $u(p)=p(1-p)$ is concave, hence
\begin{equation}\label{eq1}v(p)=\int_0^\infty re^{-rt}u(p_t^*)dt=\int_0^\infty re^{-rt} p_t^*(1-p_t^*)dt.\end{equation}
On the other hand, $p_t^*$ is given by $\displaystyle \left(\begin{array}{c}
1-p_t^* \\
p_t^*
\end{array}\right)= P(t)  \left(\begin{array}{c}
1-p \\
p
\end{array}\right).
$
Integration in (\ref{eq1}) leads to
\[v(p)=\frac14-\frac{(2p-1)^2}{4}\times \frac{r}{r+4\pi}.\]

\section{Proof of Theorem \ref{th1}}\label{sec:proofmain}

In this section, we  prove \textbf{P1}. Statement \textbf{P2} is a particular case of Theorem \ref{mainth3} below, and we postpone the proof to section \ref{sec:nonendo}. The proof of \textbf{P1} is divided in three parts. We first prove that
$$\displaystyle \liminf_n \tilde v_n(p) \geq \sup_{\calS(p)}\E\left[\int_0^\infty re^{-rt} u(p_t)dt\right], $$ and next that $$\displaystyle \limsup_n \tilde v_n(p) \leq \sup_{\calS(p)}\E\left[\int_0^\infty re^{-rt} u(p_t)dt\right].$$ We finally show that the supremum is reached.

\subsection{Step 1}

Let $(p_t)\in \calS(p)$ be arbitrary. We will prove that  $\displaystyle \liminf_n \tilde v_n(p) \geq \E\left[\int_0^\infty re^{-rt} u(p_t)dt\right]$. The proof will make use of Lemma \ref{lemm2} below. This lemma is conceptually similar to (but technically more involved than) the elementary, so-called splitting lemma (Aumann and Maschler (1995)) which we quote here.

\begin{lemma}\label{lemm1}
Let a finite set $L$,  and a probability $p\in \Delta(S)$ be given, such that $p=\displaystyle \sum_{l\in L}\alpha_lp_l$, for some $\alpha\in\Delta(L)$ and $p_l\in \Delta(S)$ ($l\in L$). Then there is a probability distribution $\prob$ over $L\times S$ with marginals given by $\alpha$ and $p$, and such that the conditional law of $s$ given $l$ is $p_l$.
\end{lemma}

The usual interpretation of Lemma \ref{lemm1} is as follows. Assume some player, informed of the realization of $s$, draws $l$ according to $\prob(l\mid s)$ and announces $l$. Then, the posterior belief of an uninformed player with prior belief $p$ is equal to $p_l$. Lemma \ref{lemm1} formalizes the extent to which an informed player can ``manipulate" the belief of an uninformed player by means of a public announcement.

 Lemma \ref{lemm2} below is the appropriate generalization of Lemma \ref{lemm1} to a dynamic world with changing states.
Some notation is required. We fix a Markov chain $(\omega_m)_{m\in \dN}$ over $S$, with initial law $p\in \Delta(S)$,  transition matrix $\Pi=(\pi(s'\mid s))_{s,s'\in S}$, and law $\prob$. Given a sequence $\mu=(\mu_m)_{m\in \dN}$, where $\mu_m$ is a transition function from $(\Delta(S))^{m}\times S$ to $\Delta(S)$,\footnote{We write $\mu_m(q^m,s;\cdot)$. Thus, $\mu_m(q^m,s;\cdot)$ is a probability distribution over $\Delta(S)$ for each given $q^m=(q_0,\ldots,q_m)\in \Delta(S)^m$, and $s\in S$, and the probability $\mu_m(q^m,s;A)$ assigned to a fixed (measurable) set $A\subset \Delta(S)$ is measurable in $(q^m,s)$.} we denote by $\mu\circ \prob$ the probability measure over $(\Delta(S)\times S)^\dN$ which is obtained as follows. Together with $\Pi$, $\mu_m$ induces a transition function $\nu_m$ from $\Delta(S)^m\times S$ to $\Delta(S)\times S$ defined by
\begin{equation}\label{eqalpha} \nu_m(q^m,\omega_m;q_{m+1},\omega_{m+1})=\pi(\omega_{m+1}\mid \omega_m)\mu_m(q^m,\omega_{m+1};q_{m+1}).\end{equation}
The distribution $\mu\circ \prob$ is the probability measure over $(\Delta(S)\times S)^\dN$ induced by the sequence $(\nu_m)_{m\in \dN}$ (by means of the Ionescu-Tulcea Theorem) and the initial distribution of $(q_0,\omega_0)$ which assigns probability $p(\omega_0)$ to $(\{p\},\omega_0)\in \Delta(S)\times S$.

To follow-up on the above interpretation, we think of an uninformed player with belief $p$ over $\omega_0$, and of an informed player who observes the successive realizations of $(\omega_m)$, and picks a new belief $q_{m+1}\in \Delta(S)$ for the uninformed player, as a (random) function $\mu_m$ of the earlier beliefs $q^m= (q_0,\ldots, q_m)$ and of the realized state $\omega_{m+1}$ in stage $m+1$. The distribution $\mu\circ \prob$ is the induced distribution over sequences of beliefs and states.

\begin{lemma}\label{lemm2}
Let  $Q$ be a probability distribution over $\Delta(S)^\dN$ such that $Q$-a.s.,  $q_0= q$ and that $\E[q_{m+1}\mid q^m]=^T\Pi q_m$ for each $m$.

Then there exists a sequence $\mu=(\mu_m)$ such that the probability measure $\mu\circ\prob$ satisfies \textbf{C1} and \textbf{C2} below.
\begin{description}
\item[C1] The marginal of $\mu\circ \prob$ over $\Delta(S)^\dN$ is $Q$.
\item[C2] For each $m\geq 0$, $q_m$ is (a version of) the conditional law of $\omega_m$ given
 $q^{m}$.
\end{description}
\end{lemma}

The proof of Lemma \ref{lemm2} is in the Appendix. We now construct a behavior strategy $\sigma_1$ of player 1 in $G_n(p)$. We let $\alpha:\Delta(S)\to \Delta(A)$ be a (measurable) function such that $\alpha(\tilde p)$ is an optimal strategy of player 1 in the one-shot, average game $\Gamma(\tilde p)$, for each $\tilde p\in \Delta(S)$.

For $k\geq 0$, we set $s_k^{(n)}:=s_{k/n}$ and $p_k^{(n)}:=p_{k/n}$. The sequence $(s_k^{(n)})_{k\in\dN}$ is a Markov chain with transition function $P_{1/n}$ and initial distribution $p$. We let $(\mu_k)_{k\in \dN}$ be the transition functions obtained by applying Lemma \ref{lemm2} with   $\omega_k:=s_k^{(n)}$ and $q_k:=p_k^{(n)}$.

According to $\sigma_1$, player 1 picks $p_{k+1}^{(n)}\in \Delta(S)$ according to $\mu_k(p_0^{(n)},\ldots, p_k^{(n)},s_{k+1}^{(n)};\cdot)$ then plays the mixed action $ \alpha(p_{k+1}^{(n)})\in \Delta(A)$.

Let $\sigma_2$ be an arbitrary strategy of player 2 in $G_n(p)$. For any given stage $k$, one has
\begin{eqnarray*}
\E[g(s_k^{(n)},a_k,b_k)] &=& \E\left[\E\left[g(s_k^{(n)},a_k,b_k)\mid p_0^{(n)},\ldots, p_k^{(n)}\right]\right]\\
&=&  \E\left[\E\left[g(s_k^{(n)},\alpha(p^{(n)}_k),b_k)\mid p_0^{(n)},\ldots, p_k^{(n)}\right]\right]\\
&=&  \E\left[\E\left[g(p_k^{(n)},\alpha(p^{(n)}_k),b_k)\mid p_0^{(n)},\ldots, p_k^{(n)}\right]\right]\\
&\geq &  \E\left[\E\left[u(p_k^{(n)})\mid p_0^{(n)},\ldots, p_k^{(n)}\right]\right] =\E[u(p_k^{(n)})].
\end{eqnarray*}

Summing over $k$, and denoting by $t\mapsto p_t^{(n)}$ the step process equal to $p_{\frac{k}{n}}$ over the interval $[\frac{k}{n},\frac{k+1}{n})$, one has therefore
\begin{equation}\label{eq3}\tilde v_n(p)\geq \E\left[\int_0^\infty re^{-rt}u(p_t^{(n)})dt\right].\end{equation}

Since $(p_t)$ is c\`adl\`ag, the map $t\mapsto p_t$ has $\prob$-a.s. at most countably many discontinuity points. Note also that $\lim_n p_t^{(n)}=p_t$ at every continuity point. Thus, one has $\lim_n p_t^{(n)}=p_t$  $\prob\otimes \textrm{Leb}$-a.s. . By dominated convergence, this implies
\[\lim_{n\to +\infty}\E\int_0^\infty re^{-rt}u(p_t^{(n)})dt=\E\int_0^\infty re^{-rt}u(p_t)dt.\]

By (\ref{eq3}), one thus has
\[\liminf_{n\to +\infty} \tilde v_n(p)\geq \E\int_0^\infty e^{-rt}u(p_t)dt.\]

\subsection{Step 2}

Let $n\in \dN^*$, and $\sigma_1$ be an arbitrary strategy of player 1. We adapt Aumann and Maschler (1995), and construct a reply $\sigma_2$ of player 2 recursively. Together with $\sigma_1$, the strategy $\sigma_2$ induces a probability distribution over plays of the game, denoted $\prob_\sigma$. Given a stage $k$, we denote by $p_k:=\Prob_\sigma(s_{k/n}=\cdot \mid \calH_k^{II})$ the belief of player 2  at the beginning  of stage $k$, where $\calH_{k}^{II}$ is the information available to player 2, that is, the $\sigma$-algebra generated by $(a_i,b_i)_{i=1,\ldots,k-1}$.\footnote{The belief $p_k$ is used to define $\sigma_2$ in stage $k$, and the computation of $p_k$ uses the definition of $\sigma_2$ in the first $k-1$ stages only. Hence, there is no circularity.} We let $\sigma_2$ play in stage $k$ a best reply in the average, one-shot game $\Gamma(p_k)$ to the conditional distribution of $a_k$ given $\calH_{k}^{II}$.

We introduce the belief  $\tilde p_k:=\Prob(s_{k/n}=\cdot \mid \calH_{k+1}^{II})$ held by player 2 at the end of stage $k$ (that is, after observing $a_k$), so that $p_{k+1}=^T\!P_{1/n}\tilde p_k$.

By Lemmas 2.5 and 2.6 in Mertens, Sorin and Zamir (1994), one has\footnote{denoting by $|\cdot |_1$ the $L^1$-norm on $\Delta(S)$.}
\[\E[g(s_{k/n},a_k,b_k)\mid \calH_{k}^{II}] \leq u(p_k)+\E\left[ |p_k-\tilde p_k|_1\mid \calH_{k}^{II}\right].\]
Taking expectations and summing over stages, one obtains
\begin{equation}\label{eq3bis}
\E\left[\lambda_n \sum_{k=0}^\infty (1-\lambda_n)^{k}g_k\right]\leq \E\left[\lambda_n\sum_{k=0}^\infty (1-\lambda_n)^{k}u(p_k)\right]+\E\left[\lambda_n\sum_{k= 0}^\infty (1-\lambda_n)^{k}|p_k-\tilde p_{k}|_1\right]
\end{equation}
We now introduce a process $(\bar p_t)$ in $\calS(p)$ defined by $\bar p_{k/n}=p_k$ and $\bar p_t=^T\!P_{t-\frac{k}{n}}\bar p_{k/n}$ for each $k$ and $t\in [\frac{k}{n},\frac{k+1}{n})$.

Choose a constant $c>0$  such that $|\bar p_t-\bar p_{k/n}|_1\leq \frac{c}{n}$ for each $k,n\in \dN^*$ and $t\in [\frac{k}{n},\frac{k+1}{n})$.
We first bound the first term on the right-hand side of (\ref{eq3bis}):
\[\E\left[ \lambda_n\sum_{k=0}^\infty(1-\lambda_n)^{k-1} u(p_k)\right]
\leq \E\left[ \int_0^\infty re^{-rt}u(\bar p_t)dt\right] +\sup_{t\geq 0} |u(\bar p_t)-u(\bar p_{\frac{1}{n}\lfloor nt\rfloor})|.\]
Since $u$ is  Lipschitz for the $L^1$-norm, one has, for some $C$,
\[\E\left[ \lambda_n\sum_{k=0}^\infty(1-\lambda_n)^{k} u(p_k)\right]\leq
\sup_{(p_t)\in \calS(p)}\E\left[  \int_0^\infty re^{-rt}u(p_t)dt\right]+\frac{C}{n}.\]
Next, adapting Mertens Sorin and Zamir (1994), one has
\begin{eqnarray*}\E\left[\lambda_n\sum_{k=0}^\infty (1-\lambda_n)^{k}|p_k-\tilde p_k|_1\right]
&=& \sum_{s\in S}\lambda_n\sum_{k=0}^\infty (1-\lambda_n)^{k}
\E\left[ |p_k(s)-\tilde p_{k}(s)|\right]\\
&\leq &  \sum_{s\in S}\left(\lambda_n\sum_{k=0}^\infty (1-\lambda_n)^{k}
\E\left[ |p_k(s)-\tilde p_{k}(s)|^2\right]\right)^{\frac{1}{2}}\\
& =&\sum_{s\in S}\left(\lambda_n\sum_{k=0}^\infty (1-\lambda_n)^{k}\left(\E(p_k^2(s))-\E(\tilde p_{k}^2(s))\right)\right)^{\frac{1}{2}}
,\end{eqnarray*}
which is also equal to
$$
 \sum_{s\in S}\left(\lambda_n\sum_{k=0}^\infty (1-\lambda_n)^{k}\left(\E(p_k^2(s))-\E(p_{k+1}^2(s))
+\E(p_{k+1}^2(s))-\E(\tilde p_{k}^2(s))\right)\right)^{\frac{1}{2}}.
$$
Therefore
\begin{eqnarray*}\E\left[\lambda_n\sum_{k=0}^\infty (1-\lambda_n)^{k}|p_k-\tilde p_k|_1\right]
&\leq &\sum_{s\in S}\left(\lambda_n+\lambda_n\sum_{k=0}^\infty (1-\lambda_n)^{k}\E\left(p_{k+1}^2(s)-\tilde p_{k}^2(s)\right)\right)^{\frac{1}{2}}\\
&\leq & |S|\left(\lambda_n +\frac{2c}{n}\right)^\frac12 .\end{eqnarray*}
Plugging into (\ref{eq3bis}), and since $\sigma_1$ is arbitrary, this yields
\[\tilde v_n(p) \leq  \sup_{(p_t)\in \calS(p)}\E\left[  \int_0^\infty re^{-rt}u(p_t)dt\right]  +\frac Cn+ |S| \left( \lambda_n +\frac{2c}{n}\right)^{\frac{1}{2}},\]
and the inequality
$\displaystyle \limsup_{n\to \infty} \tilde v_n(p)\leq \sup_{(p_t)\in \calS(p)}\E\left[  \int_0^\infty re^{-rt}u(p_t)dt\right]$ follows.

\subsection{Step 3}

We conclude by proving that the supremum in \textbf{P1} is reached.
First we remark that the claim {\bf P1} can alternatively be written as
\[ \displaystyle v(p)=\max_{\prob\in \Sigma(p)} \E_\prob\left[ \int_0^{+\infty} re^{-rt}u(p_t)dt\right],\]
where, if $\Omega$ denotes the set of  c\`adl\`ag functions from $\dR^+$ to $\Delta(S)$  and $(p_t)$ the canonical process on $\Omega$, $\Sigma(p)$ is the set of probability measures $\prob$ on $\Omega$ under which $(p_t)$ belongs to $\calS(p)$. This reformulation permits us to use classical arguments: We apply the tightness criterion of Meyer and Zheng (1984). Let $(\prob_n)_n$ be a maximizing sequence for \textbf{P1}. Although strictly speaking the coordinate process $(p_t)$ need not be a quasimartingale, Theorem 4 in Meyer and Zheng (1984) still applies.\footnote{One may e.g. consider the laws $\prob_n^T$ of the coordinate process stopped at $T$ and then use a diagonal argument. Alternatively, we may apply Theorem 4 directly to the "damped" process $q_t:=e^{-at} p_t$ where $a\in (0,r)$, with values in the cone spanned by $\Delta(S)$.}

Denote by $\bar \prob$ the weak limit of some subsequence of $(\prob_n)_n$. It is routine to show that $\bar \prob\in \calS(p)$. Finally, since the functional $\displaystyle \E\left[\int_0^\infty e^{-rt} u(p_t) dt\right]$ is weak continuous, $\bar\prob$ is a maximizer in \textbf{P1}.

\section{Games with  endogenous transitions}\label{sec:nonendo}

In this section we extend Theorem \ref{th1} to games with  endogenous transitions. We show that the limit value exists and is characterized as the unique  viscosity solution of a Hamilton-Jacobi equation.

\subsection{Model}

We now introduce a more general model in which  players control  transitions. As before, $S$ is a finite set of states,  $A$ and $B$ are  finite action sets and $g:S\times A\times B\rightarrow\dR$ denotes the payoff function. In contrast with the previous sections, we here assume that the generator depends on actions : $R:=(R(s,s';a,b),s,s'\in S,a\in A,b\in B)$, where  for all $(s,a,b)\in S\times A\times B$,
\begin{itemize}
\item for all $s'\neq s$, $R(s,s';a,b)\geq 0$,
\item $\ds \sum_{s'\in S}R(s,s';a,b)=0$.
\end{itemize}
For fixed $(a,b)\in A\times B$, we denote by $t\mapsto P_t(a,b)$ the transition semi-group of the Markov chain with transition rates $R(\cdot;a,b)$. Given $t\geq 0$, and $x\in \Delta(A)$, $y\in \Delta(B)$, we set  $\ds P_t(x,y):=\sum_{a\in A,b\in B}x(a)y(b)P_t(a,b)$.

For $n\in\dN^*$, $G_n(p)$ now  denotes the two-player game with infinitely many stages, where, at each stage $k\in\dN$, players first choose actions $a_k\in A$ and $b_k\in B$,  the payoff is $g(s^{(n)}_k,a_k,b_k)$, and next $s^{(n)}_{k+1}$ is drawn using $P_{1/n}(s^{(n)}_k,\cdot;a_k,b_k)$.

The information structure of the game is as before: player 1 observes past and current realizations of the states $s^{(n)}_k$ and both players observe past actions of their opponent, while payoffs are not observed.
As before,
the weight of stage $k$ in $G_n(p)$ is $\lambda_n(1-\lambda_n)^k$, with $\lambda_n:=\displaystyle 1-e^{-r/n}$.

The value of the game $G_n(p)$ is still denoted $\tilde v_n(p)$.

\subsection{Viscosity Solutions}

In this section, we introduce the Hamilton-Jacobi equation which characterizes the limit value and we define the notion of weak solution (in the viscosity sense) used in Theorem \ref{th1}.

We first need to fix some notations. As the partial differential equations encountered below take place in the simplex $\Delta(S)$, we have to define a tangent vector space $T_{\Delta(S)}(p)$ to the set $\Delta(S)$ at each point $p$:
$$
T_{\Delta(S)}(p):= \left\{ z=(z_s)_{s\in S} \in \dR^{|S|},\; \exists \ep>0, \; p+\ep z\in \Delta(S), \ p-\ep z\in \Delta(S) \right\}\;.
$$
For instance, if $p$ belongs to the relative interior of $\Delta(S)$, one has $T_{\Delta(S)}(p)=\{z\in \dR^{|S|}, \; \sum_{s\in S} z_s=0\}$, while $T_{\Delta(S)}(p)=\{0\}$
if $p$ is an extreme point of $\Delta(S)$. We also define, for any symmetric matrix $X\in \dR^{|S|\times|S|}$,
\be\label{deflambdamax}
\lambda_{\max}(p, X): = \sup\left\{ \frac{\lg X z,z\rg}{|z|^2}, \; z\in T_{\Delta(S)}(p)\backslash\{0\}\right\}.
\ee
By convention $\lambda_{\max}(p, X)=-\infty$ if $T_{\Delta(S)}(p)=\{0\}$. Note that $\lambda_{\max}(p, X)$ is a kind of maximal eigenvalue of $X$ on the tangent space $T_{\Delta(S)}(p)$.

Given a continuous map $H:\Delta(S)\times \dR^{|S|}\to \dR$, we consider the Hamilton-Jacobi equation
\be\label{HJ11}
\min\left\{ rw(p)+H(p,Dw(p))\ ; \ -\lambda_{\max}(p, D^2w(p))\right\}=0 \qquad {\rm in }\; \Delta(S).
\ee

\begin{definition}\label{defvisco} A map $w: \Delta(S)\to \dR$ is a {\rm viscosity subsolution} of \eqref{HJ11} if it is upper semicontinuous and if, for any smooth test function $\phi:\dR^{|S|}\to \dR$ such that $w-\phi$ has a local maximum on $\Delta(S)$ at a point $p\in \Delta(S)$, one has
$$
\min\left\{ rw(p)+H(p,D\phi(p))\ ; \ - \lambda_{\max}(p, D^2\phi(p))\right\}\leq 0 \;.
$$
A  map $w: \Delta(S)\to \dR$ is a {\rm viscosity supersolution} of \eqref{HJ11} if it is lower semicontinuous and if, for any smooth test function $\phi:\dR^{|S|}\to \dR$ such that $w-\phi$ has a local minimum on $\Delta(S)$ at a point $p\in  \Delta(S)$, one has
$$
\min\left\{ rw(p)+H(p,D\phi(p))\ ; \ - \lambda_{\max}(p, D^2\phi(p))\right\}\geq 0 \;.
$$
Finally,  $w: \Delta(S)\to \dR$ is a {\rm viscosity solution} of \eqref{HJ11} if it is a subsolution and a supersolution of \eqref{HJ11}.
\end{definition}

\begin{remark} {\rm
1) This definition does not exactly match  the standard notion of viscosity solution given, e.g., in  Crandall, Ishii and Lions (1992): the reason is that we work with functions defined on the simplex $\Delta(S)$, instead of the entire space. It is not even quite the same as in recent papers dealing with differential games with incomplete information, see e.g. Cardaliaguet and Rainer (2009a). In these papers,  no private information is ever disclosed after the initial time and the ``dynamics" on the parameter $p$ is simply the evolution of the belief of the non-informed player. As a consequence, a key property of these games is that the faces of the simplex are invariant under this dynamics: in terms of PDE, this is expressed by the fact that  the conditions for supersolution only need to be tested in the relative interior of $\Delta(S)$. In the present framework, the variable $p$ has a dynamics (the controlled Markov chain), which leaves the entire set $\Delta(S)$ invariant, but not the
  faces. As a consequence, the equations have to hold {\it up to the boundary}, as in the so-called state-constraint problems.

2) In the above definitions, one can always replace the assumption that  $w-\phi$ has a local maximum or minimum by the condition
 $w-\phi$ has a strict local maximum or minimum (see, e.g.,  Crandall et al. (1992)).
}
\end{remark}

Uniqueness for the solution of \eqref{HJ2} holds thanks to the following comparison principle, proved in the appendix. We assume that the Hamiltonian $H: \Delta(S)\times \dR^{|S|}\to \dR$ satisfies the condition:
\be\label{CondH1}
\left| H(p,q)-H(p,q')\right|�\leq C |q-q'| \qquad \forall p\in \Delta(S),\ q,q'\in \dR^{|S|}
\ee
as well as
\be\label{CondH2}
\left| H(p,q)-H(p',q)\right|�\leq C |p-p'|(|q|+1) \qquad \forall p,p'\in \Delta(S),\ q\in \dR^{|S|}
\ee

\begin{theorem}\label{theo:comparison} Under assumptions \eqref{CondH1} and \eqref{CondH2}, if $w_1$ is a continuous  viscosity subsolution of \eqref{HJ11} while $w_2$ is a continuous viscosity supersolution of \eqref{HJ11}, then $w_1\leq w_2$ in $\Delta(S)$.
\end{theorem}

In particular, equation \eqref{HJ11} has at most one continuous viscosity solution.\\

\noindent {\bf Examples:}
We have already encountered several examples of subsolution and supersolution for the Hamilton-Jacobi equation \eqref{HJ1}: for instance, it can be checked that the maps
$\ds w(p):= \int_0^\infty re^{-rt}\mathrm{cav}\ u(p^*_t)dt$ in Lemma \ref{lemmupper} is a supersolution to \eqref{HJ1}, while  the maps $\ds w_1(p):= \int_0^\infty re^{-rt}u(p^*_t)dt$ and $\ds w_2(p):= \sum_{s\in S}u(\delta_s)\int_0^\infty re^{-rt}p^*_t(s)dt$  in Lemma \ref{lemmlower} are  subsolutions of \eqref{HJ1}. Hence, by \textbf{P2} and Theorem \ref{theo:comparison}, one has $v\leq w$ and $v\geq \max(w_1,w_2)$. This provides an alternative proof of Lemmas \ref{lemmupper} and \ref{lemmlower}.\\

\subsection{The convergence result}

In the endogenous case, it seems difficult to provide a characterization of $\lim \tilde v_n$ of the type of \textbf{P1} in  Theorem \ref{th1}. However,  characterization {\bf P2} still holds:

\begin{theorem}\label{mainth3} The uniform limit
$\ds \lim_{n\to +\infty} \tilde v_n(p)$ exists and is the unique viscosity solution of the Hamilton-Jacobi equation \eqref{HJ11}, where $H$ is now given by
\be\label{DefHbis}
H(p, \xi)= \min_{x\in \Delta(A)}\max_{y\in \Delta(B)} \left\{ -\lg ^T\!R(x,y)p,\xi\rg  - r g(p,x,y)\right\}\;.
\ee
\end{theorem}

Note that, when the transition are independent of actions, one recovers statement {\bf P2} of Theorem~\ref{th1} as a particular case. As the map $H$ defined by \eqref{DefHbis} satisfies conditions \eqref{CondH1} and \eqref{CondH2} above, Theorem \ref{theo:comparison} applies, and  equation  \eqref{HJ11} has at most a unique viscosity solution.

\subsection{Proof of Theorem \ref{mainth3}}

\subsubsection{Step 1: dynamic programming principle and regularity} As  is well-known, the maps $\tilde v_n$ are (uniformly) Lipschitz on $\Delta(S)$, concave and bounded, and satisfy the following dynamic programming principle:
\be\label{PPD1}
\begin{array}{l}
\ds \tilde v_n(p)
\; = \; \ds  \max_{x\in (\Delta(A))^S}\min_{y\in \Delta(B)} \left( \lambda_n g(p,x,y)+ (1-\lambda_n)\sum_{a\in A, \ b\in B} x(p)(a) y(b) \tilde v_n(^T\!P_{\frac{1}{n}}(a,b) \hat p(x,a)) \; \right) \\
\; = \; \ds \min_{y\in \Delta(B)}\max_{x\in (\Delta(A))^S} \left( \lambda_n g(p,x,y)+ (1-\lambda_n)\sum_{a\in A, \ b\in B} x(p)(a) y(b)\tilde v_n( ^T\!P_{\frac{1}{n}}(a,b)\hat p(x,a)) \; \right)
\end{array}
\ee
where
$\ds
x(p)(a)= \sum_{s\in S} p_s x_s(a)
$
and, for  $a \in A$, $\hat p(x, a):= \left( \frac{p_s x_s(a)}{x(p)(a)}\right)_{s\in S}$ is the conditional law of the state given $a$:

We will prove that any accumulation point of $(\tilde v_n)$ for the uniform convergence is a viscosity solution of \eqref{HJ11}. Since \eqref{HJ11} has a unique viscosity solution, this will imply the uniform convergence of the sequence $(\tilde v_n)$.

We thus consider a uniformly convergent subsequence of $(\tilde v_n)$. We denote by $w$ the continuous limit, and relabel the subsequence as $(\tilde v_n)$.

\subsubsection{Step 2: $w$ is a viscosity supersolution}
Let $\phi$ be a smooth test function such that $w-\phi$ has a strict local minimum on $\Delta(S)$ at some point $\bar p\in \Delta(S)$. This implies the existence of a sequence $(p_n)$ which converges to $\bar p$ and such that $\tilde v_n-\phi$ has a local minimum at $p_n$ for any $n$: namely,
\be\label{vnphi}
\tilde v_n(p)\geq \phi(p)-\phi(p_n)+\tilde v_n(p_n)\qquad\mbox{\rm for any $p\in \Delta(S)$.}
\ee
As $\tilde v_n$ is concave, the inequality $\ds \lambda_{\max}(p_n, D^2\phi(p_n))\leq 0 $
holds by definition of $T_{\Delta(S)}(\bar p)$. Since  $T_{\Delta(S)}(p_n)\supset T_{\Delta(S)}(\bar p)$ for $n$ large enough, letting $n\to +\infty $ yields
$\ds - \lambda_{\max}(\bar p, D^2\phi(\bar p))\geq 0$.

 Let $y_n\in \Delta(B)$ achieve the minimum in \eqref{PPD1} for $\tilde v_n(p_n)$, so that
\[
 \tilde v_n(p_n)= \max_{x\in (\Delta(A))^S} \left( \lambda_n g(p_n,x,y_n)+ (1-\lambda_n)\sum_{a\in A, \ b\in B} x(p_n)(a) y_n(b)\tilde v_n( ^T\!P_{\frac{1}{n}}(a,b)\hat p_n(x,a)) \; \right)
\]
Up to a subsequence, still denoted $(y_n)$, we may assume that $(y_n)$ converges to some $y\in \Delta(B)$.

Let now $x\in \Delta(S)$ be arbitrary. Applying the latter equation with the non-revealing strategy $(x,\ldots, x)\in \Delta(A)^S$, one has (with some abuse of notation) $x(p_n)=x$ and $\hat p_n(x,a)=p_n$, hence
$$
 \tilde v_n(p_n)\geq  \lambda_n g(p_n,x,y_n)+ (1-\lambda_n)\sum_{a\in A, \ b\in B} x(a) y_n(b)\tilde v_n( ^T\!P_{\frac{1}{n}}(a,b)p_n) \;.
$$
Using \eqref{vnphi}, this yields
\be\label{ini1}
\begin{array}{rl}
\tilde v_n(p_n)\; \geq & \ds  (1-\lambda_n)\left(-\phi(p_n)+\tilde v_n(p_n)\right) \\
& \ds +  \lambda_n g(p_n,x,y_n)+ (1-\lambda_n)\sum_{a\in A, \ b\in B} x(a) y_n(b)\phi( ^T\!P_{\frac{1}{n}}(a,b) p_n)
\end{array}
\ee
Since
$P_{\frac{1}{n}}(a,b) =I+\frac{1}{n} ^T\!R(a,b)+o(\frac{1}{n})$,
$$
\phi( ^T\!P_{\frac{1}{n}}(a,b) p_n) = \phi(p_n)+ \frac{1}{n}\lg ^T\!R(a,b)p_n, D\phi(p_n)\rg + o(\frac{1}{n})\;.
$$
Since $\ds \lambda_n= 1-e^{-\frac{r}{n}}= \frac{r}{n}+o(\frac{1}{n})$, the  inequality \eqref{ini1} can then be rewritten
$$
0\; \geq \;-r \tilde v_n(p_n)+ r g(p_n,x,y_n)+ \sum_{a\in A, \ b\in B} x(a) y_n(b)\lg ^T\!R(a,b)p_n, D\phi(p_n)\rg + o(1)\;.
$$
Letting $n\to +\infty$ then yields
$
0\; \geq \;-r w(\bar p)+  rg(\bar p,x,y)+ \lg ^T\!R(x,y)\bar p, D\phi(\bar p)\rg.
$
Taking the infimum over $x\in \Delta(A)$ gives:
$$
r w(\bar p)+ \max_{y\in \Delta(B)}\min_{x\in \Delta(A)}\left\{ -  rg(\bar p,x,y)- \lg ^T\!R(x,y)\bar p, D\phi(\bar p)\rg \right\} \geq 0\;.
$$
In conclusion we have proved that
$$
\min\left\{ rw(\bar p)+H(\bar p,D\phi(\bar p))\ ; \ - \lambda_{\max}(\bar p, D^2\phi(\bar p))\right\}\geq 0 \;,
$$
as desired.

\subsubsection{Step 3: $w$ is a viscosity subsolution}

We will use the following technical remark, which follows from Cardaliaguet and Rainer (2009a) or from Step 1 in the proof of Proposition 4.4 in Gr\"{u}n  (2012):

\begin{lemma}\label{lem:tech} Let $w:\Delta(S)\to \dR$ be a concave function and $\phi$ be a smooth test function such that $w-\phi$ has a local maximum on $\Delta(S)$ at a point  $\bar p\in \Delta(S)$. If $\ds \lambda_{\max}(\bar p, D^2\phi(\bar p))<0$, then there is $\delta>0$ such that, for any $p\in \Delta(S)$ such that $p-\bar p\in T_{\Delta(S)}(\bar p)$,
\be\label{jhbqhfcb}
w(p)\leq w(\bar p)+ \lg D\phi(\bar p),p-\bar p\rg -\delta |p-\bar p|^2 .
\ee
\end{lemma}

Let $\phi$ be a smooth test function such that $w-\phi$ has a strict local maximum on $\Delta(S)$ at some point $\bar p\in \Delta(S)$. If $\ds \lambda_{\max}(\bar p, D^2\phi(\bar p))\geq 0$, then the desired inequality holds. So we may assume that $\lambda_{\max}(\bar p, D^2\phi(\bar p))<0$ and it remains to check that, in this case,
$\ds rw(\bar p)+H(\bar p,D\phi(\bar p)) \leq 0\;.$

As before there are $p_n\in \Delta(S)$ which converge to $\bar p$ and at which $\tilde v_n-\phi$ has a local maximum.
Let now $x_n=(x_{n,s})$ achieve the maximum in  \eqref{PPD1} for $\tilde v_n(p_n)$.
Given an arbitrary $y\in \Delta(B)$,  one thus has
\be\label{startingpoint}
\tilde v_n(p_n) \; \leq \;  \ds  \lambda_n g(p_n,x_n,y)+ (1-\lambda_n)\sum_{a\in A, \ b\in B} x_n(p_n)(a) y(b) \tilde v_n(^T\!P_{\frac{1}{n}}(a,b) \hat p_n(x_n,a)).
\ee
Since  $\ds \lambda_n= o(1)$, since  $^T\!P_{\frac{1}{n}}(a,b) \hat p_n(x_n,a)= \hat p_n(x_n,a)+o(1)$, and using the uniform continuity of $\tilde v_n$, this implies
$$
o(1) \; \leq  \;  \ds \sum_{a\in A, \ b\in B} x_n(p_n)(a) y(b) \left(\tilde v_n( \hat p_n(x_n,a)) -\tilde v_n(p_n)\right)
$$
Let $x=(x_s)_{s\in S}$ be the limit of (a subsequence of)  $(x_{n})_n$. Letting  $n\to+\infty$ in the above inequality we get
\be\label{sumwhat-w}
\begin{array}{rl}
0 \; \leq &  \ds \sum_{a\in A, \ b\in B} x(\bar p)(a) y(b) \left(w( \hat p(x,a)) -w(\bar p)\right) \\
=&  \ds \sum_{a\in A} x(\bar p)(a)  \left(w( \hat p(x,a)) -w(\bar p)\right)
\end{array}
\ee
As $\ds  \sum_{a\in A} x(\bar p)(a)  \hat p(x,a)= \bar p,$ all the points\footnote{such that $x(\bar p)(a)>0$} $\hat p(x,a)$ belong to the same face of $\Delta(S)$ as $\bar p$. Hence $\hat p(x,a)-\bar p\in T_{\Delta(S)}(\bar p)$ for any $a\in A$.  Since  $w-\phi$ has a local maximum on $\Delta(S)$ at $\bar p\in \Delta(S)$ with $\lambda_{\max}(\bar p, D^2\phi(\bar p))<0$, Lemma \ref{lem:tech}  states that there is $\delta>0$ such that, for all $p\in \Delta(S)$ with $p-\bar p\in T_{\Delta(S)}(\bar p)$,
\be\label{vnConv1}
w(p)\leq w(\bar p)+\lg D\phi(\bar p), p-\bar p\rg -\delta | p-\bar p|^2.
\ee
Plugging \eqref{vnConv1} into \eqref{sumwhat-w} gives
$$
\begin{array}{rl}
0 \; \leq  & \ds   \ds \sum_{a\in A} x(\bar p)(a)  \left(\lg D\phi(\bar p),  \hat p_n(x_n,a)-\bar p\rg -\delta |  \hat p_n(x_n,a)-\bar p|^2 \right) \\
& \\
= & \ds -\delta  \sum_{a\in A} x(\bar p)(a) |  \hat p_n(x_n,a)-\bar p|^2,
\end{array}
$$
because $\ds  \sum_{a\in A} x(\bar p)(a)  \hat p_n(x_n,a)= \bar p$.
In particular, $\hat p(x,a)=\bar p$  if $x(\bar p)(a) >0$.
By definition of $x(\bar p)(a)$, we have therefore  $x_s=x_{s'}$ for any $s\neq s'$ such that $\bar p_s>0$ and $\bar p_{s'}>0$ (which means that $x$ is non revealing). We denote by $x\in \Delta(A)$ this common value and note that $\hat p(x,a)=\bar p$ whenever $x(a)>0$.

We now come back to \eqref{startingpoint} and use the concavity of $\tilde v_n$ to deduce that
$$
\tilde v_n(p_n) \; \leq \;  \ds  \lambda_n g(p_n,x_n,y)+ (1-\lambda_n)\tilde v_n\left(\sum_{a\in A, \ b\in B} x_n(p_n)(a) y(b) ^T\!P_{\frac{1}{n}}(a,b) \hat p_n(x_n,a)\right)
$$
Since  $\tilde v_n-\phi$ has a  strict  local maximum at $p_n\in \Delta(S)$, we get
\begin{multline}\label{uzbechj}
0 \; \leq \;  \ds  \lambda_n \left(g(p_n,x_n,y)- \tilde v_n(p_n)\right)\\ + (1-\lambda_n)\left(\phi \left(\sum_{a\in A, \ b\in B} x_n(p_n)(a) y(b) ^T\!P_{\frac{1}{n}}(a,b) \hat p_n(x_n,a)\right) -\phi(p_n)\right)
\end{multline}
Observe next that
$$
\begin{array}{rl}
\ds \sum_{a\in A, \ b\in B} x_n(p_n)(a) y(b) ^T\!P_{\frac{1}{n}}(a,b) \hat p_n(x_n,a) \; = & \ds p_n + \frac{1}{n}
\sum_{a\in A, \ b\in B} x_n(p_n)(a) y(b) ^T\!R(a,b) \hat p_n(x_n,a)+o(\frac{1}{n})\\
= & \ds p_n + \frac{1}{n}
\sum_{a\in A, \ b\in B} x_n(p_n)(a) y(b) ^T\!R(a,b) \bar p + o(\frac1n)
\end{array}
$$
where the second equality holds because $\hat p_n(x_n,a)\to \bar p$. Plugging into in \eqref{uzbechj} we obtain
$$
0 \; \leq \;  \ds  \lambda_n \left(g(p_n,x_n,y)- \tilde v_n(p_n)\right)+ \frac{(1-\lambda_n)}{n}\big\lg D\phi(p_n), \sum_{a\in A, \ b\in B} x_n(p_n)(a) y(b) ^T\!R(a,b) \bar p\big\rg +o(\frac1n)
$$
Since $x_n(p_n)(a)\to x(a)$,   multiplying by $n$ and letting $n\to+\infty$ yields
$$
0 \; \leq \;  \ds  r \left(g(\bar p,x,y)- w(\bar p)\right)+ \big\lg D\phi(\bar p), \sum_{a\in A, \ b\in B} x(a) y(b) ^T\!R(a,b) \bar p\big\rg
$$
When rearranging, we find that
$$
rw(\bar p)+\min_{x\in \Delta(A)}\max_{y\in \Delta(B)}\left( -rg(\bar p,x,y) -  \lg D\phi(\bar p),  ^T\!R(x,y) \bar  p\rg\right)\leq 0\;.
$$
Therefore $w$ is a subsolution.

\section{Incomplete information on both sides}\label{sec:bothsides}

\subsection{Model}

The approach developed in the previous section can also be adapted to games with lack of information on both sides, in which each player observes and controls a Markov chain. The framework is  close to the one of Gensbittel and Renault (2012). In particular, we also  assume that each player  observes only one Markov chain. However, the fact that players play more and more often completely changes the nature of the results.

We  assume that there are two controlled Markov chains $(s^1_t)$ and $(s^2_t)$ with values in the finite sets  $S^1$ and $S^2$ respectively. The process $(s^i_t)$ is observed and controlled by  Player $i=1,2$. That is, the generator of $(s^1_t)_{t\geq 0}$ is of  the form  $(R^1(s,s';a), s,s'\in S^1,a\in A)$, and that of $(s^2_t)$ is   $(R^2(s,s';b), s,s'\in S^1,b\in B)$. The assumptions on $R^1$ and $R^2$ are the same as in the previous section.

Much as before, for given $a$, we  denote by $P^1_t(a)$ the transition function of a Markov chain with transition rates $R^1(\cdot;a)$,
and set $P^1_t(x):=\sum_{a\in A}x(a)P^1_t(a)$ whenever $x\in \Delta(A)$. The transition function $P^2_t(y)$ is defined similarly for $y\in \Delta(B)$.

In this new game, the payoff function depends on both states and actions: $g:S^1\times S^2\times A\times B\to \dR$. The initial positions $s^1_0$ and $s^2_0$ of the chains are chosen independently with laws $p^1\in \Delta(S)$ and $p^2\in \Delta(S^2)$.  As before,
the weight of stage $k$ in $G_n(p)$ is $\lambda_n(1-\lambda_n)^k$, with $\lambda_n:=\displaystyle 1-e^{-r/n}$.  The value of the game with initial distribution $(p^1,p^2)$  is denoted by  $\tilde v_n(p^1,p^2)$.

\subsection{The Hamilton-Jacobi equation}

In this setting, we have to introduce a slightly new type of Hamilton-Jacobi equation. The Hamiltonian  is now a map $H:\Delta(S^1)\times\Delta(S^2)\times \dR^{|S^1|+|S^2|}\to \dR$ and the Hamilton-Jacobi equation is given by the pair of inequalities
\be\label{HJ2}
\begin{array}{r}
\ds \max\left\{\min\left\{ rw+H(p^1,p^2,Dw) ;  -\lambda_{\max}(p^1, D_{11}^2w)\right\} ;   -\lambda_{\min}(p^2, D_{22}^2w)\right\}\leq0\qquad \\ \ds  {\rm in }\; \Delta(S^1)\times \Delta(S^2)\;,\\
\ds \min\left\{ \max\left\{ rw+H(p^1,p^2,Dw)  ;   -\lambda_{\min}(p^2, D_{22}^2w)\right\};  -\lambda_{\max}(p^1, D_{11}^2w)\right\}\geq 0\qquad \\ \ds  {\rm in }\; \Delta(S^1)\times \Delta(S^2)\;.
\end{array}
\ee
In the above expressions, $r>0$ is the discount rate, $w:\Delta(S^1)\times \Delta(S^2)\to \dR$ is the unknown (formally extended to a neighborhood of $\Delta(S^1)\times \Delta(S^2)$),   $Dw=(D_1w,D_2w)$ is the full gradient of $w$ with respect to $(p^1,p^2)$, $D_1w$ (resp. $D_2w$) being the derivative with respect to $p_1$ (resp.  $p_2$),  $D^2_{11}w$ (resp. $D^2_{22}w$) is the second order derivative of $w$ with respect to $p^1$ (resp. $p^2$), $\lambda_{\max}(p^1, X)$ is defined by \eqref{deflambdamax} while
\be\label{deflambdamin}
\lambda_{\min}(p^2, X) = \inf\left\{ \frac{\lg X z,z\rg}{|z|^2}, \; z\in T_{\Delta(S^2)}(p^2)\backslash\{0\}\right\}
\ee
where $T_{\Delta(S^2)}(p^2)$ is the tangent space of $\Delta(S^2)$ at $p^2$. By convention we set $\lambda_{\min}(p^2, X)=+\infty$ if $T_{\Delta(S^2)}(p^2)=\{0\}$.

As before, one cannot expect equation \eqref{HJ2} to have a smooth solution in general, and we use instead the following notion of viscosity solution:

\begin{definition} A map $w: \Delta(S^1)\times \Delta(S^2)\to \dR$ is a {\rm viscosity subsolution} of \eqref{HJ2} if it is upper semicontinuous and if, for any smooth test function $\phi:\dR^{|S^1|+|S^2|}\to \dR$ such that $w-\phi$ has a local maximum on $\Delta(S^1)\times\Delta(S^2)$ at  $(p^1,p^2)\in \Delta(S^1)\times \Delta(S^2)$, one has
$$
\begin{array}{r}
\ds \max\left\{ \min\left\{ rw(p^1,p^2)+H(p^1,p^2,D\phi(p^1,p^2))\ ;\  - \lambda_{\max}(p^1, D^2_{11}\phi(p^1,p^2))\right\}\ ;\right.\qquad \\
\ds \left.- \lambda_{\min}(p^2, D^2_{22}\phi(p^1,p^2))\right\} \leq 0 \;.
\end{array}
$$
A  map $w: \Delta(S^1)\times \Delta(S^2)\to \dR$ is a {\rm viscosity supersolution} of \eqref{HJ2} if it is lower semicontinuous and if, for any smooth test function $\phi:\dR^{|S^1|+|S^2|}\to \dR$ such that $w-\phi$ has a local minimum on $\Delta(S^1)\times \Delta(S^2)$ at  $(p^1,p^2)\in  \Delta(S^1)\times \Delta(S^2)$, one has
$$
\begin{array}{r}
\ds \min\left\{\max \left\{ rw(p^1,p^2)+H(p^1,p^2,D\phi(p^1,p^2))\ ;\ - \lambda_{\min}(p^2, D^2_{22}\phi(p^1,p^2))\right\} \ ;\right.\qquad \\
\ds  \left. - \lambda_{\max}(p^1, D^2_{11}\phi(p^1,p^2))\right\}\geq 0 \;.
\end{array}
$$
Finally,  $w: \Delta(S^1)\times \Delta(S^2)\to \dR$ is a {\rm viscosity solution} of \eqref{HJ2} if it is a sub- and a supersolution of \eqref{HJ2}.
\end{definition}

Uniqueness of a viscosity solution for \eqref{HJ2} holds thanks to a comparison principle, which generalizes Theorem \ref{theo:comparison}. We will assume that $H: \Delta(S^1)\times \Delta(S^2)\times \dR^{|S^1|+|S^2|}\to \dR$ satisfies the condition:
\be\label{CondH11}
\left| H(p,q)-H(p,q')\right|�\leq C |q-q'| \qquad \forall p\in \Delta(S^1)\times \Delta(S^2),\ q,q'\in \dR^{|S^1|+|S^2|}
\ee
as well as
\be\label{CondH21}
\left| H(p,q)-H(p',q)\right|�\leq C |p-p'|(|q|+1) \qquad \forall p,p'\in \Delta(S^1)\times \Delta(S^2),\ q\in \dR^{|S^1|+|S^2|}
\ee

\begin{theorem}\label{theo:comparison2} Assume that \eqref{CondH11} and \eqref{CondH21} hold. Let $w_1$ be a  continuous viscosity subsolution of \eqref{HJ2} and $w_2$ be a continuous viscosity supersolution of \eqref{HJ2}. Then $w_1\leq w_2$ in $\Delta(S^1)\times \Delta(S^2)$.
\end{theorem}

\subsection{The limit theorem}

Here is our main result in the framework of lack of information on both sides.
\begin{theorem}\label{mainthboth} The uniform limit
$\ds v= \lim_{n\to +\infty} \tilde v_n$ exists and is the unique viscosity solution of Hamilton-Jacobi equation \eqref{HJ2}  where  $H$ is given by
\be\label{DefH}
\begin{array}{rl}
\ds H(p^1,p^2, \xi^1,\xi^2)\; = & \ds \min_{x\in \Delta(A)}\max_{y\in \Delta(B)} \left\{ -\lg ^T\!R^1(x)p^1,\xi^1\rg -\lg ^T\!R^2(y)p^2,\xi^2\rg  - r g(p^1,p^2,x,y)\right\}\\
=& \ds \max_{y\in \Delta(B)}\min_{x\in \Delta(A)} \left\{ -\lg ^T\!R^1(x)p^1,\xi^1\rg -\lg ^T\!R^2(y)p^2,\xi^2\rg  - r g(p^1,p^2,x,y)\right\}
\end{array}
\ee
for any $(p^1,p^2, \xi^1,\xi^2)\in \Delta(S^1)\times \Delta(S^2)\times \dR^{|S^1|}\times \dR^{|S^2|}$.
\end{theorem}

 Note that the Hamiltonian defined in \eqref{DefH} satisfies conditions \eqref{CondH11} and \eqref{CondH21}. So equation \eqref{HJ2} has at most one viscosity solution.

\subsection{Proof of Theorem \ref{mainthboth}}

The proof is close to the one for Theorem \ref{mainth3}. The main difference is that we have to deal with the fact that both players now have private information, which complicates the proof of the viscosity solution property. On the other hand, the problem is now symmetrical, so that it is enough to show the supersolution property, the argument for the subsolution being identical.

\subsubsection{Step 1: Dynamic programming principle and regularity }

For $(p^1,p^2)\in \Delta(S^1)\times \Delta(S^2)$,  $(x,y)\in (\Delta(A))^{|S^1|}\times (\Delta(B))^{|S^2|}$ and $(a,b)\in A\times B$, we set
$$
\ds
x(p^1)(a)= \sum_{s\in S^1} p^1_s x_s(a)\ {\rm and }\ y(p^2)(b)= \sum_{s\in S^2} p^2_s y_s(b)
$$
and $\displaystyle \hat p^1(x, a):= \left( \frac{p_s^1 x_s(a)}{x(p^1)(a)}\right)_{s\in S^1}$ and $\displaystyle \hat p^2(y, b):
= \left( \frac{p^2_s y_s(b)}{x(p^2)(b)}\right)_{s\in S^2}$
denote the conditional distributions of the states given $a$ and $b$ respectively.

The dynamic programming principle for  $\tilde v_n$ reads
\be\label{PPD4}
\begin{array}{l}
\ds \tilde v_n(p^1,p^2) \\
\; = \; \ds  \max_{x\in (\Delta(A))^{|S^1|}}\min_{y\in (\Delta(B))^{|S^2|}} \left( \lambda_n g(p^1,p^2,x,y)\right.\\
\qquad \ds  \left.+ (1-\lambda_n)\sum_{a\in A, \ b\in B} x(p^1)(a) y(p^2)(b) \tilde v_n(^T\!P^1_{\frac{1}{n}}(a) \hat p^1(x,a), ^T\!P^2_{\frac{1}{n}}(b) \hat p^2(y,b)) \; \right) \\
\; = \; \ds  \min_{y\in (\Delta(B))^{|S^2|}} \max_{x\in (\Delta(A))^{|S^1|}}\left( \lambda_n g(p^1,p^2,x,y)\right.\\
\qquad \ds  \left.+ (1-\lambda_n)\sum_{a\in A, \ b\in B} x(p^1)(a) y(p^2)(b) \tilde v_n(^T\!P^1_{\frac{1}{n}}(a) \hat p^1(x,a), ^T\!P^2_{\frac{1}{n}}(b) \hat p^2(y,b)) \; \right)
\end{array}
\ee

As before, the maps $\tilde v_n$ are uniformly Lipschitz and bounded, and we will prove that any (uniform) accumulation point of the sequence $(\tilde v_n)$   is a viscosity solution of \eqref{HJ2}. Again up to a subsequence, we may assume that $(\tilde v_n)$ converges to some continuous map $w$.

\subsubsection{Step 2: $w$ is a viscosity supersolution}
Let $\phi$ be a smooth test function such that $w-\phi$ has a strict local minimum on $\Delta(S^1)\times\Delta(S^2)$ at some point $(\bar p^1,\bar p^2)\in \Delta(S^1)\times \Delta(S^2)$. From standard arguments, this implies the existence of  a sequence $(p^1_n,p^2_n)$ which converges to $(\bar p^1,\bar p^2)$ and such that $\tilde v_n-\phi$ has a local minimum at $(p^1_n,p^2_n)$ for any $n$: namely,
\be\label{vnphi24}
\tilde v_n(p^1,p^2)\geq \phi(p^1,p^2)-\phi(p^1_n,p^2_n)+\tilde v_n(p^1_n,p^2_n)\qquad\mbox{\rm for any $(p^1,p^2)\in \Delta(S^1)\times \Delta(S^2)$.}
\ee
As $\tilde v_n$ is concave in $p^1$, we must have
$\ds \lambda_{\max}(p^1_n, D^2_{11}\phi(p^1_n,p^2_n))\leq 0$
by definition of $T_{\Delta(S^1)}(p^1_n)$. Since $T_{\Delta(S^1)}(p^1_n)\supset T_{\Delta(S^1)}(\bar p^1)$ for $n$ large enough,
we get,  as $n\to +\infty$:
$$\ds - \lambda_{\max}(\bar p^1, D^2_{11}\phi(\bar p^1,\bar p^2))\geq 0.$$
It remains to check that
$$
\max \left\{ rw(\bar p^1,\bar p^2)+H(\bar p^1,\bar p^2,D\phi(\bar p^1,\bar p^2))\ ;\ - \lambda_{\min}(\bar p^2, D^2_{22}\phi(\bar p^1,\bar p^2))\right\} \geq 0\;.
$$
For this we assume that $\ds \lambda_{\min}(\bar p^2, D^2_{22}\phi(\bar p^1,\bar p^2))>0$ and we are left to prove that
$$
rw(\bar p^1,\bar p^2)+H(\bar p^1,\bar p^2,D\phi(\bar p^1,\bar p^2)) \geq 0\;.
$$
 Let $y_n\in (\Delta(B))^{|S^2|}$ be optimal in the dynamic programming equation \eqref{PPD4} for $\tilde v_n(p_n)$:
$$
\begin{array}{l}
\ds \tilde v_n(p^1_n,p^2_n) \; = \; \ds   \max_{x\in (\Delta(A))^{|S^1|}}\left( \lambda_n g(p^1_n,p^2_n,x,y_n)\right.\\
\qquad \ds  \left.+ (1-\lambda_n)\sum_{a\in A, \ b\in B} x(p_n^1)(a) y_n(p_n^2)(b) \tilde v_n(^T\!P^1_{\frac{1}{n}}(a) \hat p_n^1(x,a), ^T\!P^2_{\frac{1}{n}}(b) \hat p_n^2(y,b)) \; \right) \;.
\end{array}
$$
Let $y=(y_s)_{s\in S^2}$ be the limit of (a subsequence of) $(y_{n})$.
Fix $x\in \Delta(A)$. (With a slight abuse of notation),  if Player 1 plays the non-revealing strategy $(x,\dots,x)\in (\Delta(A))^{|S^1|}$, we get $x(p^1_n)=x$ and $\hat p^1_n(x,a)=p^1_n$ and therefore
\be\label{point1}
\begin{array}{rl}
\ds \tilde v_n(p^1_n,p^2_n) \; \geq &\ds    \lambda_n g(p^1_n,p^2_n,x,y_n)
\\
&\ds \qquad
+ (1-\lambda_n)\sum_{a\in A, \ b\in B} x(a)y_n(p_n^2)(b) \tilde v_n(^T\!P^1_{\frac{1}{n}}(a) p_n^1, ^T\!P^2_{\frac{1}{n}}(b) \hat p_n^2(y,b)) \;.
\end{array}
\ee
Recalling that $\ds \lambda_n= o(1)$, that $^T\!P^1_{\frac{1}{n}}(a,b) p_n^1= p_n^1+o(1)$ and that $^T\!P^1_{\frac{1}{n}}(a,b) \hat p^2_n(y_n,b)= \hat p^2_n(y_n,b)+o(1)$, we get, letting $n\to+\infty$ in \eqref{point1},
\be\label{jhbjkhzec}
\ds w(\bar p^1,\bar p^2) \; \geq \; \sum_{ b\in B} y(\bar p^2)(b) w(\bar p^1,  \hat p^2(y,b)) \;.
\ee
From \eqref{jhbjkhzec} we conclude as in the proof of Theorem \ref{mainth3} that  $ \hat p^2(y,b)=\bar p^2$ if $y(p^2)(b)>0$.
Coming back to the definition of $y(p^2)(b)>0$, we have therefore that  $y_s=y_{s'}$ for any $s\neq s'$ such that $\bar p_s^2>0$ and $\bar p_{s'}^2>0$: this means that $y$ is non revealing. We denote by $y\in \Delta(B)$ this common value and note that $\hat p^2(y,b)=\bar p$ whenever $y(b)>0$.

With this in mind, we come back to \eqref{point1}, which  becomes, since $\tilde v_n$ is convex in $p^2$, and  since the dynamics of  $(s^1_t)$ is independent of Player~2:
$$
\begin{array}{rl}
\ds \tilde v_n(p^1_n,p^2_n) \; \geq &\ds    \lambda_n g(p^1_n,p^2_n,x,y_n)
\\
&\ds \qquad
+ (1-\lambda_n)\sum_{a\in A} x(a)\tilde v_n\left(^T\!P^1_{\frac{1}{n}}(a) p_n^1, \sum_{ b\in B} y_n(p_n^2)(b) ^T\!P^2_{\frac{1}{n}}(b) \hat p_n^2(y,b)\right) \;.
\end{array}
$$
We next use the fact that $\tilde v_n-\phi$ has a local minimum  at $( p_n^1,p_n^2)$:
\be\label{jhbzdcljn}
\begin{array}{rl}
\ds 0 \; \geq &\ds    \lambda_n \left(g(p^1_n,p^2_n,x,y_n)-\tilde v_n(p^1_n,p^2_n) \right)
\\
&\ds \qquad
+ (1-\lambda_n)\sum_{a\in A} x(a)\left(\phi(^T\!P^1_{\frac{1}{n}}(a) p_n^1, \sum_{ b\in B} y_n(p_n^2)(b) ^T\!P^2_{\frac{1}{n}}(b) \hat p_n^2(y,b))
-\phi(p^1_n,p^2_n) \right)
\end{array}
\ee
where
$$
^T\!P^1_{\frac{1}{n}}(a) p_n^1= p_n^1 +\frac1n\ ^T\!R^1(a) p_n^1+o(\frac{1}{n})
$$
while, as $\ds \sum_{ b\in B} y_n(p_n^2)(b)  \hat p_n^2(y,b)= p_n^2$,
$$
\sum_{ b\in B} y_n(p_n^2)(b) ^T\!P^2_{\frac{1}{n}}(b) \hat p_n^2(y,b)=
p_n^2 + \frac1n\sum_{ b\in B} y_n(p_n^2)(b)\ ^T\!R^2(b) \hat p_n^2(y,b)+o(\frac{1}{n}).
$$
Multiplying \eqref{jhbzdcljn} by $n$ and letting $n\to+\infty$ gives therefore
$$
\begin{array}{rl}
\ds 0 \; \geq &\ds    r \left(g(\bar p^1,\bar p^2,x,y)-w(\bar p^1,\bar p^2) \right)
+ \sum_{a\in A} x(a)\big \lg D\phi(\bar p^1,\bar p^2), \left(^T\!R^1(a)\bar p^1, \sum_{ b\in B} y(b)\ ^T\!R^2(b) \bar  p^2\right)\rg\;.
\end{array}
$$
Rearranging we find that
$$
rw(\bar p)+\min_{x\in \Delta(A)}\max_{y\in \Delta(B)}\left( -g(\bar p^1,\bar p^2,x,y) -  \lg D\phi(\bar p),  (^T\!R^1(x) \bar  p^1, ^T\!R^2(y) \bar  p^2)\rg\right)\geq 0\;.
$$
Therefore $w$ is a supersolution.

\appendix

\section{Technical results}
\subsection{Proof of Lemma \ref{lemm2}}

Let a probability measure $Q$ over  $\Delta(S)^dN$ be as stated. Generic elements of $(\Delta(S)\times S)^\dN$ are denoted $(q_m,s_m)_{m\in \dN}$. To avoid multiplying notations, and at the cost of a notational abuse, we will write $Q(q^m;q_{m+1})$ for the conditional law of $q_{m+1}$ given $q^m$.
Given a probability measure $\tilde \prob$ over $(\Delta(S)\times S)^{\dN}$, we similarly write $\tilde \prob(q^m,s_m;q_{m+1}, s_{m+1})$ for the law of $(q_{m+1},s_{m+1})$ given $(q^m,s_m)$,  $\tilde \prob(q^m,s_m,s_{m+1};q_{m+1})$ for the law of $q_{m+1}$ given $(q^m,s_m,s_{m+1})$, etc., with  semi-colons separating conditioning variables from the others.

For $m\geq 1$, denote by $\theta_m$ the transition function from
$\Delta(S)^m\times S$ to $\Delta(S)\times S$ defined by
\begin{equation}\label{beta}\theta_m(q^m;F,s_{m+1})=\int_F q_{m+1}(s_{m+1}) Q(q^m;dq_{m+1}).\end{equation}
 Intuitively, $\theta_m(q^m;q_{m+1},s_{m+1})$ is the probability obtained  when first choosing $q_{m+1}$ according to its (conditional) law $Q(q^m;q_{m+1})$, then picking $s_{m+1}$ according to $q_{m+1}$, and we define the sequence $\mu=(\mu_m)$ by
\[\mu_m(q^m,s_{m+1};F):=\frac{\theta_m(q^m;F,s_{m+1})}{\theta_m(q^m;s_{m+1})},\]
so that $\mu_m(q^m,s_{m+1};q_{m+1})$ is ``the conditional law of $q_{m+1}$ given $q^m$ and $s_{m+1}$".

\bigskip

We now prove by induction  that the induced distribution $\mu\circ \prob$ over $(\Delta(S)\times S)^\dN$ satisfies \textbf{C1} and \textbf{C2}. We thus assume that for some $m$, (i) the conditional law of $q_m$ given $q^{m-1}$ (under $\mu\circ\prob$) is equal to $Q(q^{m-1};q_m)$, (ii) the conditional law of $s_m$ given $q^{m}$ is equal to $q_m$, and prove that (i) and (ii) also hold for $m+1$.

For $F\subseteq \Delta(S)$, note first that by (ii), one has
\[\mu\circ \prob(q^m;q_{m+1}\in F)=\sum_{s_m\in S} \mu\circ\prob(q^m,s_m;F)q_m(s_m)=
\sum_{s_m,s_{m+1}\in S}   \mu\circ\prob(q^m,s_m;F,s_{m+1})q_m(s_m) .\]
Since $\mu\circ \prob(q^m,s_m;F,s_{m+1})=\mu_m(q^m,s_{m+1};F)\pi(s_{m+1}\mid s_m)$ by definition, one also has
\[
\mu\circ \prob(q^m; q_{m+1}\in F)= \sum_{s_m,s_{m+1}}  \frac{\theta_m(q^m;F,s_{m+1})}{\theta_m(q^m;s_{m+1})}q_m(s_m)\pi(s_{m+1}\mid s_m) .
\]

Observe next that, by the induction assumption,  since $\E[q_{m+1}\mid q^m]=^T\Pi q_m$ and since $s_{m+1}$ and $q^m$ are conditionally independent given $s_m$, one has
\[\theta_m(q^m;s_{m+1})=\sum_{s_m}q_m(s_m)\pi(s_{m+1}\mid s_m).\]
Hence
\begin{equation}\label{point}\mu\circ \prob(q^m;F)=\sum_{s_{m+1}}\theta_m(q^m;s_{m+1},F)=Q(q^m;F),\end{equation}
as desired.

\bigskip

We now prove that $\mu\circ\prob(q^m,q_{m+1};s_{m+1})=q_{m+1}(s_{m+1})$. One has (beware of the semi-colons)
\begin{equation}\label{eq2}
\mu\circ\prob(q^m,q_{m+1};s_{m+1})=\frac{\mu\circ\prob(q^m;q_{m+1},s_{m+1})}{\mu\circ\prob(q^m;q_{m+1})}=\frac{\mu\circ\prob(q^m,s_{m+1};q_{m+1})\mu\circ\prob(q^m;s_{m+1})}{Q(q^m;q_{m+1})}
\end{equation}
using (conditional) Bayes laws and (\ref{point}). Observe next that
\begin{eqnarray*}
\mu\circ\prob(q^m,s_{m+1};q_{m+1})&=&\mu_m(q^m,s_{m+1};q_{m+1}) \\
&=& \frac{\theta_m(q^m;q_{m+1},s_{m+1})}{\theta_m(q^m;s_{m+1})}
\end{eqnarray*}
and
\begin{eqnarray*}
\mu\circ\prob (q^m;s_{m+1})&=& \sum_{s_m}\mu\circ\prob (q^m;s_m,s_{m+1})\\
&=&\sum_{s_m}\mu\circ\prob (q^m;s_m) ,s_{m+1})\times \mu\circ\prob (q^m,s_m;s_{m+1})\\
&=&\sum_{s_m} q_m(s_m)\pi(s_{m+1}\mid s_m) = \E[q_{m+1}(s_{m+1})\mid q^m],
\end{eqnarray*}
where the second equality holds by (\ref{eqalpha}).
Plugging into (\ref{eq2}), this yields
\[\mu\circ\prob(q^m,q_{m+1};s_{m+1})=\frac{\theta_m(q^m;q_{m+1},s_{m+1})}{\theta_m(q^m;s_{m+1})}
\times \frac{\E[q_{m+1}(s_{m+1})\mid q^m]}{\theta_m(q^m;q_{m+1})}.\]
To conclude, recall that (see (\ref{beta}))
\[\theta_m(q^m;q_{m+1},s_{m+1})= \theta_m(q^m;q_{m+1})q_{m+1}(s_{m+1}),\]
while
\begin{eqnarray*}
\theta_m(q_m;s_{m+1}) &=& \int_{\Delta(S)}q_{m+1}(s_{m+1})Q(q^m;dq_{m+1}) = \E[q_{m+1}(s_{m+1})\mid q^m],
\end{eqnarray*}
so that $\mu\circ \prob(q^m,q_{m+1};s_{m+1})=q_{m+1}(s_{m+1})$, as desired.
\subsection{Proof of the comparison principle}

In this section we prove Theorem \ref{theo:comparison2} (which implies  Theorem \ref{theo:comparison}). We follow here the proof of Crandall, Ishii and Lions (1992)  for second order Hamilton-Jacobi equations. However their results do not apply directly, because  the terms $\lambda_{\min}$ and $\lambda_{\max}$ introduce a strong degeneracy in the equation. This issue is also present in Cardaliaguet and Rainer (2009a), where it is dealt with by an induction argument over the dimension of the faces of the  simplices, which relies on the fact that the restriction of solutions to  faces are still solutions. This is no longer the case here. This forces us to revisit the proof, and  to come back to the basic technique consisting in regularizing the solutions by inf- or sup convolution, and then in using Jensen Lemma. \\

Let  $w_1$ be a subsolution and $w_2$ be a supersolution of \eqref{HJ2}. Our aim is to show that $w_1\leq w_2$. We argue by contradiction, and assume that
\be\label{Contradiction}
M:=\sup_{p\in \Delta(S^1)\times\Delta(S^2)} \left\{ w_1(p)-w_2(p)\right\}\; >\; 0\;.
\ee
In order to use the special structure of the problem, we have to regularize the maps $w_1$ and $w_2$ by sup and inf-convolution respectively. This technique is standard and we refer to Crandall, Ishii and Lions (1992) for details. For $\delta>0$ and $p\in \dR^{|S^1|+|S^2|}$ we set
$$
w_1^\delta(p)= \max_{p'\in \Delta(S^1)\times\Delta(S^2)} \left\{ w_1(p')-\frac{1}{2\delta}|p-p'|^2\right\}
$$
and
$$
w_{2,\delta}(p)= \min_{p'\in \Delta(S^1)\times\Delta(S^2)} \left\{ w_2(p')+\frac{1}{2\delta}|p-p'|^2\right\}
$$
We note for later use that $w_1^\delta$ and $w_{2,\delta}$ are now defined over the entire space $\dR^{|S^1|+|S^2|}$, that $w_1^\delta$ is semiconvex while $w_{2,\delta}$ is semiconcave (see Crandall et al (1992)). Moreover,
\be\label{growthw1w2}
\lim_{|p|\to+\infty} |p|^{-1} w_1^\delta(p)=-\infty, \; \lim_{|p|\to+\infty}|p|^{-1} w_{2,\delta}(p)=+\infty\;.
\ee
Setting
\be\label{Mdelta}
M_\delta= \sup_{p\in \dR^{|S^1|+|S^2|}}  \left\{ w_1^\delta(p)-w_{2,\delta}(p)\right\}
\ee
we have:
\begin{lemma}\label{lem:limpdelta} For any $\delta>0$, the problem \eqref{Mdelta} has at least one maximum point. If $p_\delta$ is such a maximum point and if $p'_\delta\in  \Delta(S^1)\times\Delta(S^2)$ and $p''_\delta\in  \Delta(S^1)\times\Delta(S^2)$ are such that
\be\label{pprimepseconde}
(i)\; w_1^\delta(p_\delta)= w_1(p'_\delta)-\frac{1}{2\delta}|p_\delta-p'_\delta|^2\qquad {\rm and}\qquad
(ii)\; w_{2,\delta}(p_\delta)= w_2(p''_\delta)+\frac{1}{2\delta}|p_\delta-p''_\delta|^2
\ee
then, as $\delta \to 0$, $M_\delta \to M$ while
$$
\frac{|p_\delta-p'_\delta|^2}{2\delta}+ \frac{|p_\delta-p''_\delta|^2}{2\delta} \to 0.
$$
\end{lemma}

\begin{proof} The existence of a maximum point is a straightforward consequence of \eqref{growthw1w2}.  The rest of the statement is classical.
\end{proof}

Next we note that $w_1^\delta$ and $w_{2,\delta}$ are still respectively a subsolution and a supersolution of slightly modified equations:
\begin{lemma}\label{lem:ineqw1w2} Assume that $v_1^\delta$ has a second order Taylor expansion at a point $p$. Then
\be\label{ineqw1}
 \min\left\{ rw_1^\delta(p)+H(p',Dw_1^\delta(p))\ ;\  - \lambda_{\max}((p')^1, D^2_{11}w_1^\delta(p))\right\} \leq 0,
 \ee
 where $p'=((p')^1,(p')^2)\in \Delta(S^1)\times\Delta(S^2)$ is such that
 $$
 w_1^\delta(p)= w_1(p')-\frac{1}{2\delta}|p-p'|^2.
 $$
Similarly, if $w_{2,\delta}$ has a second order  Taylor expansion at a point $p$, then
\be\label{ineqw2}
\max \left\{ rw_{2,\delta}(p)+H(p'',Dw_{2,\delta}(p))\ ;\ - \lambda_{\min}((p'')^2, D^2_{22}w_{2,\delta}(p)\right\}  \geq 0,
 \ee
 where $p''=((p'')^1,(p'')^2)\in \Delta(S^1)\times\Delta(S^2)$ is such that $$
 w_{2,\delta}(p)= w_2(p'')+\frac{1}{2\delta}|p-p''|^2.
 $$
\end{lemma}

\begin{proof} We do the proof for $w_1^\delta$, the argument for $w_{2,\delta}$ being symmetrical.
Assume that  $w_1^\delta$ has a second order  Taylor expansion at a point $\bar p$ and set, for $\gamma>0$ small,
$$
\phi_\gamma(p)= \lg  Dw_1^\delta(\bar p), p-\bar p\rg +\frac12\lg D^2w_1^\delta(\bar p)(p-\bar p),p-\bar p\rg + \frac{\gamma}{2}|p-\bar p|^2.
$$
We also denote by $\bar p'$ a point in $\Delta(S^1)\times \Delta(S^2)$ such that
\be\label{lienbarpbarpprime}
 w_1^\delta(\bar p)= w_1(\bar p')-\frac{1}{2\delta}|\bar p-\bar p'|^2.
 \ee
Then $w_1^\delta-\phi_\gamma$ has a maximum at $\bar p$, which implies, by definition of $w_1^\delta$, that
$$
w_1(p') -\frac{1}{2\delta}|p'-p|^2 \leq \phi_\gamma(p)-\phi_\gamma(\bar p)+ w_1^\delta(\bar p)\qquad \forall p\in \dR^{|S^1|\times|S^2|}, \ p'\in \Delta(S^1)\times \Delta(S^2),
$$
with an equality for $(p,p')=(\bar p,\bar p')$. If we choose $p=p'-\bar p'+\bar p$ in the above formula, we get:
$$
w_1(p') \leq \phi_\gamma(p'-\bar p'+\bar p)+\frac{1}{2\delta}|\bar p'-\bar p|^2 -\phi_\gamma(\bar p)+ w_1^\delta(\bar p)\qquad \forall p'\in \Delta(S^1)\times \Delta(S^2),
$$
with an equality at $p'=\bar p'$. As $w_1$ is a subsolution, we obtain therefore, using the right-hand side of the above inequality as a test function,
$$
 \min\left\{ rw_1(\bar p')+H(\bar p',D\phi_\gamma(\bar p))\ ;\  - \lambda_{\max}((\bar p')^{1}, D^2_{11}\phi_\gamma(\bar p))\right\} \leq 0.
$$
We note that $D\phi_\gamma(\bar p)= Dw_1^\delta(\bar p)$, $D^2_{11}\phi_\gamma(\bar p)= D^2w_1^\delta(\bar p)+\gamma I$ and
$w_1(\bar p')\geq w_1^\delta(\bar p)$ (by \eqref{lienbarpbarpprime}). So letting $\gamma\to 0$ we obtain the desired result.
\end{proof}

In order to exploit inequalities \eqref{ineqw1} and \eqref{ineqw2}, we have to produce points at which $w_1^\delta$ is strictly concave with respect to the first variable while $w_2^\delta$ is strictly convex with respect to the second one. For this reason, we introduce a new penalization. For $\sigma>0$ and $p=(p^1,p^2)\in \dR^{|S^1|+|S^2|}$, let us set $\bar \xi_1(p^1)= (1+|p^1|^2)^\frac12$, $\bar \xi_2(p^2)= (1+|p^2|^2)^\frac12$
and
$$
M_{\delta,\sigma} = \sup_{p\in \dR^{|S^1|+|S^2|}}  \left\{ w_1^\delta(p)-w_{2,\delta}(p)+ \sigma \bar \xi_1( p^1) +\sigma \bar \xi_2( p^2)\right\}
$$
Using \eqref{growthw1w2}, one easily checks that there exists a maximizer $(\hat p^1,\hat p^2)$ to the above problem. In order to use Jensen's Lemma (Lemma A.3. in  Crandall, Ishii and Lions  \cite{CIL}), we need  this maximum to be strict. For this we modify slightly $\bar \xi_1$ and $\bar \xi_2$: we set, for $i=1,2$, $\xi_i(p^i)= \bar \xi_i(p^i)- \sigma(1+ |p^i-\hat p^i|^2)^{\frac12}$. We will choose $\sigma>0$ so small that $\xi_1$ and $\xi_2$ still have a positive second order derivative. By definition,
$$
M_{\delta,\sigma} = \sup_{p\in \dR^{|S^1|+|S^2|}}  \left\{ w_1^\delta(p)-w_{2,\delta}(p)+ \sigma \xi_1( p^1) +\sigma \xi_2( p^2)\right\}
$$
and the above problem has a strict maximum at $(\hat p^1,\hat p^2)$.
As the map $p\to w_1^\delta(p)-w_{2,\delta}(p)+ \sigma \xi_1( p^1) +\sigma \xi_2( p^2)$ is semiconcave, Jensen's Lemma
states that, for any $\ep>0$, there is vector $a_\ep\in \dR^{|S^1|+|S^2|}$ with $|a_\ep|\leq \ep$, such that the problem
$$
M_{\delta,\sigma,\ep}:=\sup_{p\in \dR^{|S^1|+|S^2|}}  \left\{ w_1^\delta(p)-w_{2,\delta}(p)+\sigma \xi_1( p^1) +\sigma \xi_2( p^2)+ \lg a_\ep,p\rg \right\}
$$
has a maximum point $p_{\delta,\sigma,\ep}\in \dR^{|S^1|+|S^2|}$ at which the maps $w_1^\delta$ and $w_{2,\delta}$ have a second order  Taylor expansion.
 From Lemma \ref{lem:ineqw1w2}, we have
\be\label{ineqw1delta}
 \min\left\{ rw_1^\delta(p_{\delta,\sigma,\ep})+H(p_{\delta,\sigma,\ep}',Dw_1^\delta(p_{\delta,\sigma,\ep}))\ ;\  - \lambda_{\max}((p_{\delta,\sigma,\ep}')^1, D^2_{11}w_1^\delta(p_{\delta,\sigma,\ep}))\right\} \leq 0,
\ee
and
\be\label{ineqw2delta}
\max \left\{ rw_{2,\delta}(p_{\delta,\sigma,\ep})+H(p_{\delta,\sigma,\ep}'',Dw_{2,\delta}(p_{\delta,\sigma,\ep}))\ ;\ - \lambda_{\min}((p_{\delta,\sigma,\ep}'')^2, D^2_{22}w_{2,\delta}(p_{\delta,\sigma,\ep})\right\}  \geq 0,
\ee
where $p'_{\delta,\sigma,\ep}$ and $p''_{\delta,\sigma,\ep}$ are points in $\Delta(S^1)\times \Delta(S^2)$ at which one has
 $$
 w_1^\delta(p_{\delta,\sigma,\ep})= w_1(p'_{\delta,\sigma,\ep})-\frac{1}{2\delta}|p_{\delta,\sigma,\ep}-p_{\delta,\sigma,\ep}'|^2\; {\rm and }\;
 w_{2,\delta}(p_{\delta,\sigma,\ep})= w_2(p''_{\delta,\sigma,\ep})+\frac{1}{2\delta}|p_{\delta,\sigma,\ep}-p''_{\delta,\sigma,\ep}|^2 .
 $$
Note for later use that
\be\label{Dw1=Dw2=}
Dw_1^\delta(p_{\delta,\sigma,\ep})=-\frac{1}{\delta}\left(p_{\delta,\sigma,\ep}-p_{\delta,\sigma,\ep}'\right)\;{\rm and} \;
D w_{2,\delta}(p_{\delta,\sigma,\ep})= \frac{1}{\delta}\left(p_{\delta,\sigma,\ep}-p_{\delta,\sigma,\ep}''\right).
\ee
By definition of $M_{\delta,\sigma,\ep}$ we have
$$
w_1^\delta(p)\leq M_{\delta,\sigma,\ep}+ w_{2,\delta}(p) -\sigma (\xi_1( p^1)+\xi_2( p^2)) \lg a_\ep,p\rg \qquad \forall p\in \dR^{|S^1|+|S^2|},
$$
with an equality at $p_{\delta,\sigma,\ep}$. Hence
\be\label{eqDw1Dw2}
Dw_1^\delta(p_{\delta,\sigma,\ep})
= Dw_{2,\delta}(p_{\delta,\sigma,\ep}) -\sigma  \left(\begin{array}{c}D_1\xi_1(p^1_{\delta,\sigma,\ep})\\D_2\xi_2(p^2_{\delta,\sigma,\ep})\end{array} \right)- a_\ep
\ee
while
\be\label{ineqD2w1vsD2w2}
D^2w_1^\delta(p_{\delta,\sigma,\ep})
\leq D^2w_{2,\delta}(p_{\delta,\sigma,\ep}) -\sigma \left(\begin{array}{cc}D^2_{1,1}\xi_1(p^1_{\delta,\sigma,\ep})&0\\0&D^2_{2,2}\xi_2(p^2_{\delta,\sigma,\ep})\end{array} \right).
\ee
We now check that $\ds \lambda_{\max}((p_{\delta,\sigma,\ep}')^1,D^2_{11}w_1^\delta(p_{\delta,\sigma,\ep}))<0$. For this, we come back to the definition of $w_{2,\delta}$ and note that, for any $p^1\in \dR^{|S^1|}$ and $(p')^1\in \Delta(S^1)$,
$$
w_{2,\delta}(p^1,p_{\delta,\sigma,\ep}^2)\leq w_2((p')^1,(p_{\delta,\sigma,\ep}')^2)+\frac{1}{2\delta}\left(|p^1-(p')^1|^2+
|(p_{\delta,\sigma,\ep})^2-(p_{\delta,\sigma,\ep}')^2|^2\right),
$$
with an equality at $(p^1, (p')^1)= ((p_{\delta,\sigma,\ep})^1, (p_{\delta,\sigma,\ep}')^1)$. If $z\in T_{\Delta(S^2)}(p_{\delta,\sigma,\ep}')^1$ with $|z|$ small enough, taking $p^1:= (p_{\delta,\sigma,\ep})^1+ z$ and $(p')^1= (p_{\delta,\sigma,\ep}')^1+ z$ gives
\begin{multline*}
w_{2,\delta}((p_{\delta,\sigma,\ep})^1+ z,p_{\delta,\sigma,\ep}^2)\leq \\ w_2((p_{\delta,\sigma,\ep}')^1+ z,(p_{\delta,\sigma,\ep}')^2)+\frac{1}{2\delta}\left(|(p_{\delta,\sigma,\ep})^1-(p_{\delta,\sigma,\ep}')^1|^2+
|(p_{\delta,\sigma,\ep})^2-(p_{\delta,\sigma,\ep}')^2|^2\right),
\end{multline*}
with equality for $z=0$. As $w_2$ is concave with respect to the first variable, the above inequality implies that  $\ds \lambda_{\max}((p_{\delta,\sigma,\ep}')^1,D^2_{11}w_{2,\delta}(p_{\delta,\sigma,\ep}))\leq 0$. In view of \eqref{ineqD2w1vsD2w2} we get therefore
$$
\ds \lambda_{\max}((p_{\delta,\sigma,\ep}')^1,D^2_{11}w_1^\delta(p_{\delta,\sigma,\ep})) \leq
-\sigma \lambda_{\min}((p_{\delta,\sigma,\ep}')^1,D^2_{1,1}\xi_1(p^1_{\delta,\sigma,\ep}))<0
$$
because $D^2_{1,1}\xi_1>0$ by contruction.
One can check in the same way that $$ \lambda_{\min}((p_{\delta,\sigma,\ep}'')^2, D^2_{22}w_{2,\delta}(p_{\delta,\sigma,\ep})>0.$$ So \eqref{ineqw1delta} and \eqref{ineqw2delta} become
\be\label{ineqw1deltaNew}
rw_1^\delta(p_{\delta,\sigma,\ep})+H(p_{\delta,\sigma,\ep}',Dw_1^\delta(p_{\delta,\sigma,\ep})) \leq 0
\ee
and
\be\label{ineqw2deltaNew}
 rw_{2,\delta}(p_{\delta,\sigma,\ep})+H(p_{\delta,\sigma,\ep}'',Dw_{2,\delta}(p_{\delta,\sigma,\ep}))  \geq 0.
\ee
We compute the difference of the two inequalities:
$$
r(w_1^\delta(p_{\delta,\sigma,\ep})-w_{2,\delta}(p_{\delta,\sigma,\ep}))
+H(p_{\delta,\sigma,\ep}',Dw_1^\delta(p_{\delta,\sigma,\ep})) -H(p_{\delta,\sigma,\ep}'',Dw_{2,\delta}(p_{\delta,\sigma,\ep})) \leq 0,
$$
where, in view of assumption \eqref{CondH21} and \eqref{Dw1=Dw2=},
$$
H(p_{\delta,\sigma,\ep}',Dw_1^\delta(p_{\delta,\sigma,\ep}))
\geq
H(p_{\delta,\sigma,\ep},Dw_1^\delta(p_{\delta,\sigma,\ep})) -\frac{C}{\delta}\left|p_{\delta,\sigma,\ep}-p_{\delta,\sigma,\ep}'\right|^2
$$
while
$$
H(p_{\delta,\sigma,\ep}'',Dw_{2,\delta}(p_{\delta,\sigma,\ep}))\leq
H(p_{\delta,\sigma,\ep},Dw_{2,\delta}(p_{\delta,\sigma,\ep}))+ \frac{C}{\delta}\left|p_{\delta,\sigma,\ep}-p_{\delta,\sigma,\ep}''\right|^2.
$$
So
$$
\begin{array}{l}
\ds r(w_1^\delta(p_{\delta,\sigma,\ep})-w_{2,\delta}(p_{\delta,\sigma,\ep}))
+H(p_{\delta,\sigma,\ep},Dw_1^\delta(p_{\delta,\sigma,\ep})) -H(p_{\delta,\sigma,\ep},Dw_{2,\delta}(p_{\delta,\sigma,\ep})) \\
\qquad \qquad \ds \leq
 \frac{C}{\delta}\left(\left|p_{\delta,\sigma,\ep}-p_{\delta,\sigma,\ep}'\right|^2+\left|p_{\delta,\sigma,\ep}-p_{\delta,\sigma,\ep}''\right|^2\right).
\end{array}$$
We now use assumption \eqref{CondH11} on $H$ combined with \eqref{eqDw1Dw2} to deduce:
\be\label{ihbdczehjb}
\ds r(w_1^\delta(p_{\delta,\sigma,\ep})-w_{2,\delta}(p_{\delta,\sigma,\ep}))
\leq
 \frac{C}{\delta}\left(\left|p_{\delta,\sigma,\ep}-p_{\delta,\sigma,\ep}'\right|^2+\left|p_{\delta,\sigma,\ep}-p_{\delta,\sigma,\ep}''\right|^2\right)+C(\ep+\sigma),
\ee
since $D\xi_1$ and $D\xi_2$ are bounded. As $\sigma$ and $\ep$ tend to $0$, the $p_{\delta,\sigma,\ep}$, $p_{\delta,\sigma,\ep}'$ and $p_{\delta,\sigma,\ep}''$ converges (up to a subsequence) to  $p_\delta$, $p_\delta'$ and $p_\delta''$, where $p_\delta$ is a maximum in \eqref{Mdelta} and where $p_\delta'$ and $p_\delta''$ satisfy \eqref{pprimepseconde}. Moreover \eqref{ihbdczehjb} implies that
$$
\ds rM_\delta= r(w_1^\delta(p_\delta)-w_{2,\delta}(p_\delta))
\leq
 \frac{C}{\delta}\left(\left|p_\delta-p_\delta'\right|^2+\left|p_\delta-p_\delta''\right|^2\right).
$$
We finally let $\delta\to 0$: in view of Lemma \ref{lem:limpdelta} the above inequality yields to $M=\lim_{\delta\to0}M_\delta\leq 0$, which contradicts our initial assumption. Therefore $w_1\leq w_2$ and the proof is complete.\\

\noindent{\bf Acknowledgments.} We thank Fabien Gensbittel for fruitful discussions.\\

\noindent This work has been partially supported by the French National Research Agency
 ANR-10-BLAN 0112.

\end{document}